\documentclass[12pt]{amsart}

\makeatletter

\def\chaptermark#1{}

\def\chapter{%
  \if@openright\cleardoublepage\else\clearpage\fi
  \thispagestyle{plain}\global\@topnum\z@
  \@afterindenttrue \secdef\@chapter\@schapter}

\def\@chapter[#1]#2{\refstepcounter{chapter}%
  \ifnum\c@secnumdepth<\z@ \let\@secnumber\@empty
  \else \let\@secnumber\thechapter \fi
  \typeout{\chaptername\space\@secnumber}%
  \def\@toclevel{0}%
  \ifx\chaptername\appendixname \@tocwriteb\tocappendix{chapter}{#2}%
  \else \@tocwriteb\tocchapter{chapter}{#2}\fi
  \chaptermark{#1}%
  \addtocontents{lof}{\protect\addvspace{10\p@}}%
  \addtocontents{lot}{\protect\addvspace{10\p@}}%
  \@makechapterhead{#2}\@afterheading}
\def\@schapter#1{\typeout{#1}%
  \let\@secnumber\@empty
  \def\@toclevel{0}%
  \ifx\chaptername\appendixname \@tocwriteb\tocappendix{chapter}{#1}%
  \else \@tocwriteb\tocchapter{chapter}{#1}\fi
  \chaptermark{#1}%
  \addtocontents{lof}{\protect\addvspace{10\p@}}%
  \addtocontents{lot}{\protect\addvspace{10\p@}}%
  \@makeschapterhead{#1}\@afterheading}
\newcommand\chaptername{Chapter}

\def\@makechapterhead#1{\global\topskip 7.5pc\relax
  \begingroup
  \fontsize{\@xivpt}{18}\bfseries\centering
    \ifnum\c@secnumdepth>\m@ne
      \leavevmode \hskip-\leftskip
      \rlap{\vbox to\z@{\vss
          \centerline{\normalsize\mdseries
              \uppercase\@xp{\chaptername}\enspace\thechapter}
          \vskip 3pc}}\hskip\leftskip\fi
     #1\par \endgroup
  \skip@34\p@ \advance\skip@-\normalbaselineskip
  \vskip\skip@ }
\def\@makeschapterhead#1{\global\topskip 7.5pc\relax
  \begingroup
  \fontsize{\@xivpt}{18}\bfseries\centering
  #1\par \endgroup
  \skip@34\p@ \advance\skip@-\normalbaselineskip
  \vskip\skip@ }
\def\appendix{\par
  \c@chapter\z@ \c@section\z@
  \let\chaptername\appendixname
  \def\thechapter{\@Alph\c@chapter}}

\newcounter{chapter}

\newif\if@openright

\makeatother

\usepackage{amsmath}
\usepackage{tikz-cd}
\usepackage{mathtools}
\usepackage{amsfonts}
\usepackage{graphicx}
\usepackage{verbatim}
\usepackage{textcomp}\usepackage{amssymb}
\usepackage{cite}
\usepackage{color}
\usepackage[all]{xy}
\usepackage{graphicx}
\usepackage{color}
\usepackage{hyperref}
\usepackage{mathrsfs}
\usepackage{soul}
\usepackage[normalem]{ulem}
\hypersetup{
    colorlinks,
    citecolor=black,
    filecolor=black,
    linkcolor=black,
    urlcolor=black
}
\usepackage{tikz}
\usepackage[all]{xy}
\usepackage{graphicx}
\usetikzlibrary{backgrounds}
\newtheorem{theorem}{Theorem}[section]
\newtheorem{lemma}[theorem]{Lemma}
\newtheorem{proposition}[theorem]{Proposition}
\newtheorem{corollary}[theorem]{Corollary}   

\newtheorem{claim}[theorem]{Claim}
\allowdisplaybreaks

\usepackage{amsthm}

\theoremstyle{definition}
\newtheorem{definition}[theorem]{Definition}
\newtheorem{remark}[theorem]{Remark}

 \newenvironment{proof_of_main_theorem2}
{{\it Proof of  Theorem~\ref{maintheorem}.}}
{\hfill $\Box$ \\}

\newenvironment{proof_of_main_theorem3}
{{\it Proof of  Theorem~\ref{maintheorem'}.}}
{\hfill $\Box$ \\}

\newenvironment{prooflemma:star'}
{{\it Proof of  Lemma~\ref{lemma:star'}.}}
{\hfill $\Box$ \\}

\newenvironment{proof1}
{{\it Proof of  Theorem~\ref{maintechnical}.}}
{\hfill $\Box$ \\}

\numberwithin{equation}{section}

\newcommand{\R}{\mathbb R}

\newcommand{\N}{\mathbb N}
\newcommand{\C}{\mathbb C}

\newcommand{\Sp}{\mathbb S}

\newcommand{\B}{B}

\newcommand{\deltae}{\delta_{\text{euc}}}

\newcommand{\ee}{{(\epsilon_j)}}

\newcommand{\PF}{P_F}
\newcommand{\FY}{F \times Y}
\newcommand{\PFk}{P_{F_k}}

\newcommand{\Waff}{W_{\mathrm{aff}}}

\newcommand\wlim{\mathop{\mbox{$\omega$-$\lim$}}}
\title[Harmonic maps into Euclidean buildings]{Harmonic maps into Euclidean buildings and non-Archimedean superrigidity}

\thanks{
CB supported in part by NSF DMS CAREER-1750254, CM supported in part by NSF DMS-2304697.}

\author[Breiner]{Christine Breiner}
\address{Brown University\\
Department of Mathematics\\
Providence, RI}
\email{christine\underline{ }breiner@brown.edu}

\author[Dees]{Ben K. Dees}
\address{Brown University\\
Department of Mathematics\\
Providence, RI}
\email{benjamin\underline{ }dees@brown.edu}

\author[Mese]{Chikako Mese}
\address{Johns Hopkins University\\
Department of Mathematics\\
Baltimore, MD}
\email{cmese@math.jhu.edu}

\begin{document}

\maketitle

\begin{abstract}
We prove that harmonic maps into Euclidean buildings, which are not necessarily locally finite, have singular sets of Hausdorff codimension 2, extending the locally finite regularity result of Gromov and Schoen. As an application, we prove superrigidity for {algebraic groups over fields with non-Archimedean valuation}, thereby generalizing the rank 1 $p$-adic superrigidity results of Gromov and Schoen and casting the Bader-Furman generalization of Margulis' higher rank superrigidity result in a geometric setting.  We also prove an existence theorem for a pluriharmonic map from a K\"ahler manifold to a Euclidean building.
\end{abstract}

\section{Introduction}\label{sec:intro}

Gromov and Schoen’s celebrated result [GS] established $p$-adic superrigidity and the consequent arithmeticity for lattices of certain rank 1 groups. Alongside Corlette’s rank 1 Archimedean superrigidity result [C], these findings complement Margulis’ higher rank superrigidity results \cite{margulis}. Both Corlette and Gromov-Schoen's theorems are geometric superrigidity results.  In other words, they determine conditions under which an isometric action on a complete CAT(0) space has a fixed point or leaves a convex subset invariant. The proof is through harmonic map techniques, with Gromov-Schoen’s proof notably involving the generalization of classical harmonic map techniques to a singular setting.

A Euclidean building $X$ is equipped with a distance function $d$ which makes $(X,d)$ into a Hadamard space (ie~a complete metric space satisfying CAT(0) triangle comparison).  Euclidean buildings share similarities with  Riemannian symmetric spaces of non-compact type which make them natural subjects of geometric study. Kleiner and Leeb [KL] illustrated this connection by showing that asymptotic cones of Riemannian symmetric spaces are Euclidean buildings.

More broadly, J. Tits introduced buildings to provide a geometric interpretation of a certain class of  groups \cite{tits}.
Specifically, given a semi-simple group $G$ over
a field endowed with a non-Archimedean valuation, Tits constructed a  metric space $X$  where $G$ acts by isometries.    
Special cases are the  Bruhat-Tits buildings associated with $p$-adic Lie groups, which are distinguished by the property that they are locally finite Euclidean buildings \cite{bruhat-tits}. 

The technical achievement of \cite{gromov-schoen} lies in  developing a harmonic map theory applicable to singular spaces which include Bruhat-Tits buildings {for $p$-adic groups}. In particular, Gromov and Schoen establish that the singular set of a harmonic map into {such a building} is small. This enables them to utilize non-linear Bochner techniques developed by Siu \cite{siu} and Corlette \cite{corlette} {for maps into smooth manifolds} to prove rigidity results.
 {(See \cite{freidin,freidin-zhang,zhang-zhong-zhu} for Bochner formulas for harmonic maps into singular spaces.)} 
To the best knowledge of the authors, harmonic map theory remains the only known method for proving superrigidity in rank 1 cases.

When the valuation is not discrete, the work of \cite{gromov-schoen} no longer applies. In particular, analyzing harmonic maps into the associated Euclidean building becomes challenging due to the absence of local finiteness in the space.  
The main technical result of this paper is the following.  When $\dim X=1$,  this theorem is due to   \cite{sun}.

\begin{theorem} \label{maintechnical}
If $u:\Omega \rightarrow X$ is a harmonic map from a smooth Riemannian domain with Lipschitz boundary into a Euclidean building (not necessarily locally finite), then the singular set of $u$ is {a closed set} of Hausdorff codimension 2. 
\end{theorem}
\noindent {The {\it singular set} is the set of points satisfying the property that none of its neighborhoods is mapped into a single apartment.  See Definition~\ref{def:singular} for the precise definition.} {Note that in the locally finite setting, for $\text{dim}(\Omega)=n$, Dees \cite{dees} recently proved the stronger result that the singular set is in fact $(n-2)$-rectifiable.}

 As a consequence of Theorem \ref{maintechnical}, we obtain the following rigidity theorem.  For lattices in groups of rank 1, this  generalizes the geometric superrigidity results of  Corlette (cf~\cite{corlette}) and Gromov-Schoen (cf~\cite{gromov-schoen}).   For  lattices in groups of rank $\geq 2$, this constitutes the geometric superrigidity assertion corresponding to the group theoretic statement of Bader-Furman \cite{bader-furman} proved through the   dynamics of semisimple Lie groups. 
 
  \begin{theorem} \label{maintheorem}
 Let $\widetilde M= G \slash K$ be an irreducible symmetric space of noncompact type that is not the Euclidean space, 
 $SO_0(p,1) \slash SO(p) \times SO(1)$, nor $SU_0(p,1) \slash S(U(p) \times U(1))$.
  Let 
 $\Gamma$ be a lattice in $G$  and let $\rho: \Gamma \rightarrow \text{Isom}(X)$ be a homomorphism where $X$ is a Euclidean building (not necessarily locally finite), {and $\rho(\Gamma)$ does not fix a point at infinity}. If the rank of $\widetilde M$ is $\geq 2$, we additionally assume that $\Gamma$ is cocompact. 
 Then $\rho(\Gamma)$ fixes a point of $X$.
  \end{theorem}
   
Next, we consider the case when the domain is a K\"ahler manifold.  Following~\cite[Section 7]{gromov-schoen},  a harmonic map from a K\"ahler manifold to a Euclidean building is called {\it pluriharmonic} if it is pluriharmonic in the usual sense {away from the singular set. }

\begin{theorem}\label{maintheorem'} Let $\widetilde M$ be the universal cover of a complete finite volume K\"ahler manifold $(M, \omega)$. Let 
$\Gamma =\pi_1(M)$, $X$  a Euclidean building (not necessarily locally finite) and $\rho:\Gamma \rightarrow \text{Isom}(X)$ a group homomorphism. 
Then any finite energy  $\rho$-equivariant harmonic map $u: \widetilde M \rightarrow X$ is pluriharmonic.  {In particular, if there exists a  $\rho$-equivariant map of finite energy into $X$, then there exists a $\rho$-equivariant pluriharmonic map into $X$.}
\end{theorem}

Theorem~\ref{maintheorem} is a further extension of the development of harmonic map techniques in singular spaces  to address geometric superrigidity problems.
Pioneered by \cite{gromov-schoen}, these singular spaces encompass  locally finite hyperbolic buildings  (cf~\cite{daskal-meseGAFA}), the Weil-Petersson completion of Teichm\"uller space (cf~\cite{daskal-meseINV}) and now, non-simplicial Euclidean buildings.  

Theorem~\ref{maintheorem'} builds upon the work of  Gromov-Schoen (cf~\cite[Section 9]{gromov-schoen}), which asserts the existence of pluriharmonic maps in the singular setting.
A notable application is in the factorization theorems for  Zariski dense representations of  fundamental groups of algebraic varieties.
In the paper \cite{corlette-simpson}, Corlette and Simpson proved that such a representation into $\mathsf{SL}_2(\C)$ factors through an orbicurve if it is non-rigid or not integral.   The rigidity aspect can be interpreted as the statement that a representation   into $\mathsf{SL}_2(\C(t))$  goes into a compact subgroup. The core argument in their proof involves harmonic maps to the Bruhat-Tits building (also known as the Serre tree) of $\mathsf{SL}_2(\mathbb{C}(t))$.  Since the Serre tree for $\mathsf{SL}_2(\C(t))$ is not locally compact, Gromov-Schoen theory does not apply, and they have to  make  a reduction mod $p$ to the case of representations in $\mathsf{SL}_2(\mathbb{F}_p((t)))$ where $\mathbb{F}_p$ is a finite field.  Hence,  this part of their paper can be greatly simplified by Theorem 1.3 or by Sun’s [Su] treatment of harmonic maps into $\R$-trees.  Our theorem can further be exploited in the generalizations to non-locally finite buildings that have been studied by various authors (cf [KNPS1], [KNPS2] and references therein).  Additionally, Theorem 1.3 can be applied to study factorization theorems for higher rank local system.

All of these theorems expand the scope of harmonic map techniques by allowing general Euclidean building targets.  An important example is   the harmonic map  that appears as the limit of a sequence of rescaled maps associated to a sequence of harmonic maps into a symmetric space with unbounded energy.  This is an   important tool in understanding the compactification of representation varieties (cf~\cite{wolf}, \cite{ddw},  
\cite{loftin-tamburelli-wolf}).

\subsection{Main Ideas}

Let $\Omega$ be a Lipschitz Riemannian domain and let $(X,d)$ be a Euclidean building.  We review concepts from \cite{gromov-schoen} that generalize analytical notions for real-valued functions.  These key concepts are also used in our paper.
\begin{itemize}
\item
{\it order of a harmonic map $u:\Omega \rightarrow (X,d)$ at $x_0 \in \Omega$} (cf~\cite[Section 2]{gromov-schoen}): For a harmonic function $f$, the order  at $x_0$ is the degree of the dominant homogeneous harmonic polynomial approximating $f(x)-f(x_0)$ near $x_0$.   
\item    {\it homogeneous degree 1 maps into $X$} (cf~\cite[Section 3]{gromov-schoen}): These are map with the property  that the restriction  to a radial ray   is a constant speed geodesic.   Homogeneous degree 1 harmonic maps generalize affine functions.  
\item  {\it instrinsically differentiable maps into $X$}(cf~\cite[Section 5]{gromov-schoen}):    These are maps that can be approximated near a point  by a homogeneous degree 1 harmonic error term going to zero faster than distance to the point.  Intrinsically differentiable maps generalize differentiable functions.  
\item {\it blow up maps $u_\sigma$ of a harmonic map $u:\Omega \rightarrow (X,d)$ at $x_0 \in \Omega$} (cf~\cite[Section 3]{gromov-schoen}):  For $\sigma>0$ small, restrict $u$ to $B_\sigma(x_0)$ and rescale in the domain by a factor of $\sigma$ with respect to normal coordinates centered at $x_0$ and rescale the distance function of $X$ by an appropriate constant $\mu_\sigma$ dependent on $\sigma$ to construct  $u_\sigma:B_1(0) \rightarrow (X,\mu_\sigma^{-1} d)$.  Blow up maps generalize the difference quotients of functions.
\end{itemize}

We now revisit the main components in the proof of Gromov-Schoen's regularity statement. The assumption that $X$ is locally finite plays a pivotal role, presenting a challenge when extending the analysis to general Euclidean buildings.  The two key components of Gromov and Schoen's proof are:

\begin{itemize}
\item[(1)]
{\it A harmonic map $u:\Omega \rightarrow (X,d)$ near $x_0 \in \Omega$ is approximated by its tangent map $u_*$ at $x_0$.} For a locally finite building $X$, there is a neighborhood $U$ of a point $u(x_0) \in X$ and a neighborhood $V$ of the vertex of the tangent cone $T_{u(x_0)}X$ such that $U$ and $V$ are isometric.  Using this fact,  Gromov and Schoen can assume that the blow up maps $u_\sigma$ and $u_*$ map  into a metric cone.  Indeed, identifying $U$ with $V$, the rescaling of the distance function is equivalent to rescaling the cone.  Thus, they can assume that for all $\sigma>0$ sufficiently small, $u_\sigma$  maps into the  tangent cone $T_{u(x_0)}X$. Applying Arzela-Ascoli, they take the limit of a subsequence $u_k:=u_{\sigma_k}$ to obtain a tangent map $u_*:B_1(0) \rightarrow T_{u(x_0)}X$.  
\item[(2)] 
{\it The  tangent map $u_*$ at an order 1 point is effectively contained in a product space $\R^m \times Y$ where $Y$ is a lower dimensional Euclidean building.} For an order 1 point $x_0$ of $u$, a tangent map $u_*$ of $u$ at $x_0$ is a homogeneous degree 1 harmonic map and its image is a flat $F$,~ie~a copy of Euclidean space $\R^m$ isometrically and totally geodesically embedded in $T_{u(x_0)}X$.  The union $P_F$ of all apartments (ie~all top dimensional  flats) containing $F$ is a subbuilding of $T_{u(x_0)}X$ which  is isometric to $\R^m \times Y$ where $Y$ is a Euclidean building of dimension $N-m$.   The simplicial structure of apartments in a locally finite building implies that $u_*$ is then {\it effectively contained} in $P_F$ (cf~\cite[Section 5]{gromov-schoen}).  See Figure~\ref{fig:phi1}.
\end{itemize}

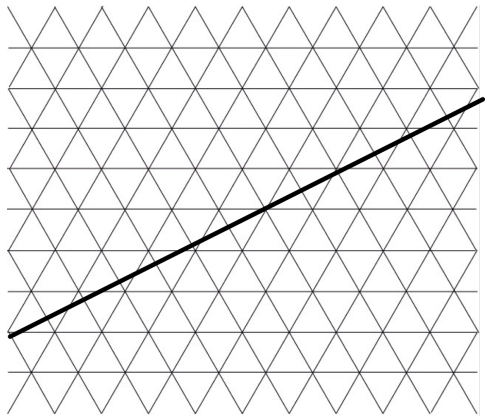
\begin{figure}[ht]
\centering

\begin{tikzpicture}[scale=0.7, x=1.15cm, y=1.15cm]
\pgfmathsetmacro{\rt}{1.7320508}    
\pgfmathsetmacro{\hh}{0.8660254}    

\clip (0,0) rectangle (10,8);

\foreach \k in {0,...,9} {
  \draw[gray!70, line width=0.5pt] (0,{\k*\hh}) -- (10,{\k*\hh});
}

\foreach \m in {-20,-19,...,20} {
  \draw[gray!70, line width=0.5pt]
    (-2,{\rt*(-2) + \m*2*\hh}) -- (12,{\rt*(12) + \m*2*\hh});
  \draw[gray!70, line width=0.5pt]
    (-2,{-\rt*(-2) + \m*2*\hh}) -- (12,{-\rt*(12) + \m*2*\hh});
}

%
%

\draw[line width=0.8mm, line cap=round]
  (0.05,1.6) -- (1.6,2.3) -- (4.2,3.6) -- (6.3,4.7) -- (10.2,6.6);

\end{tikzpicture}
  \caption{In the locally finite case, a homogeneous degree 1  map $L$ is effectively contained in  $P_F$.  This diagram depicts an example when $P_F$ is an apartment.  The thin lines represent walls of an apartment and the thick line represent $F=L(\R^n)$. The map $L$ is effectively contained in $P_F$ since the preimage of walls is locally a finite set of hyperplanes. Therefore,  the set of points mapping close to the complement of $P_F$ is small since apartments only  intersect along walls.  For the non-locally finite case, the situation is much more complicated.  For example, the union of walls could be a dense set in an apartment.}
  \label{fig:phi1}
   \label{diagram:effective}
\end{figure}

Using the above ingredients, Gromov and Schoen prove a regularity theorem \cite[Thoerem 6.3]{gromov-schoen} by an inductive argument based on the dimension of $X$.   The main step of the regularity theorem is \cite[Theorem 5.1]{gromov-schoen} which shows that, near an order 1 point, the image of a  harmonic map  is contained in the product space $\R^m \times Y$ from (2). Therefore,  a harmonic map locally decomposes into two harmonic maps, one into $\R^m$ and another into $Y$.
From this, they conclude that  the singular set  in a  neighborhood of an order 1 point is of Hausdorff codimension 2 by the inductive hypothesis.  Thus, by proving that the set of higher order points is of Hausdorff codimension 2,  they complete the  proof of  Theorem~\ref{maintechnical} for Bruhat-Tits buildings.  
Since the Gromov-Schoen theory exclusively addresses locally finite simplicial complexes, it cannot be directly applied to analyze the non-simplicial target spaces considered in this paper.

The core of this paper  involves proving the {\it local product structure} of a harmonic map at an order 1 point  described in the above paragraph for a general Euclidean building.  To do so,  we capitalize  on the Euclidean building structure to replace the reliance on the local finiteness and simplicial structure.  An $N$-dimensional Euclidean building is a union of apartments, ie~isometrically  embedded copies of Euclidean space $\R^N$.  Moreover, $\R^N$ comes equipped with an affine Weyl group, ie~a subgroup of the isometry group of $\R^N$ generated by reflections  across hyperplanes, and such that its  rotational part is a finite reflection group $W$.   The group $W$ plays an especially important role in our paper, which we highlight here.

First, because  $X$ is not necessarily locally compact, we cannot assume that the blow up maps $u_\sigma$ centered at a  point $x_0$ all have the same target space.  Hence, instead of employing the conventional limit of a sequence of maps, we rely on the ultralimit $u_\omega=\wlim u_k$ (or alternatively, the Korevaar-Schoen limit $u_*$ defined in \cite[Section 3]{korevaar-schoen2}) of blow up maps $u_k:=u_{\sigma_k}$.  The target space of $u_\omega$ is  the ultralimit $X_\omega:=\wlim X_k$, where $X_k: =(X,\mu_{\sigma_k}^{-1}d)$ are the rescalings of $X$. By \cite[Theorem 5.5.1]{kleiner-leeb}, $X_\omega$ is also a Euclidean building of type $W$.   If $x_0$ is an order 1 point, then $u_\omega$ is a homogeneous degree 1 harmonic map that can be extended to a map $L:\R^n \rightarrow X_\omega$ whose image is a flat $F$ contained in an apartment $A$. We use the map $L$  to ``pull back" $u_\omega$ to $X_k$ and  construct   a homogeneous degree 1 harmonic map $L_k$ which is close to  $u_k$.  The subbuilding $P_{F_k}$ associated to the flat $F_k=\mbox{image}(L_k)$ is isometric to a product $\R^m \times Y_k$ where the building structure of $Y_k$ depends on $W$.

Second, because $X_k$ is not necessarily locally finite,  we cannot assume that $L_k$ is  effectively contained in the subbuilding $\PFk$.  To overcome this difficulty, we use the building structure of $X$ that governs the way in which two apartments  intersect.  Indeed, the angle of intersection of two apartments is restricted by the finite reflection group $W$ that defines the building structure of $X$, and hence of $X_k$.  The main idea of this paper is to show that,  using this structure, if the image of a harmonic map is not contained in $\PFk$,  there is a significant loss of energy when we project that harmonic map into $\PFk$.  Since $u_k$ is close to $L_k$ and $\mbox{image}(L_k)\subset \PFk$, such a loss of energy contradicts the energy minimality of $u_k$.  The crux of this paper is a careful analysis of the projection map $\pi_k:X_k \rightarrow \PFk$ in order to derive a quantitative estimate of the energy loss when composing $\pi_k$ and  $u_k$.

\subsection{Organization of this paper}
Section~\ref{sec:prelim} provides references for concepts most relevant to this paper.  These are:
\begin{itemize}
\item  \cite{gromov-schoen},  \cite{korevaar-schoen1}, \cite{korevaar-schoen2} for harmonic map theory into CAT(0) spaces,
\item \cite{kleiner-leeb} for Euclidean buildings, and 
\item \cite{kleiner-leeb}, \cite{lytchak} for ultralimits of metric spaces and Lipschitz maps.
\end{itemize}
Essential details are  briefly summarized in that section.  We recommend having these references readily available when reading this paper.

Section~\ref{sec:proj} discusses  the  subbuilding $\PF$ of a Euclidean building $X$  defined as a union of all apartments parallel to a given flat $F$. The closest point projection map $\pi: X \to \PF$ is carefully analyzed.  We apply this analysis to show that a composition $\pi \circ u$ of the projection map $\pi: X \rightarrow \PF$ with a harmonic map $u$ into $X$ results in a loss of energy.

Section~\ref{sec:noregMeas} is the technical core of the paper. Expanding on the results of Section 3, we demonstrate that a harmonic map $u$ that is close to a homogeneous degree 1 harmonic map $L$ takes most points in the domain  into the subbuilding $\PF$ defined by the flat $F=\text{image}(L)$.  Specifically, we show that ``closeness in $C^0$" implies ``closeness in measure"; ie,~the set of points not mapping to the subbuilding $\PF$ via $u$ can be made arbitrarily small by assuming that $u$ is sufficiently close to a homogeneous degree 1 map $L$. This is analogous to key component (2) of Gromov-Schoen's proof, extended to the general setting considered here.

In Section~\ref{sec:approx}, we construct a sequence of homogeneous degree 1 harmonic maps that approximates a harmonic map at a point.  The argument presented in this section adapts key component (1) of Gromov-Schoen's argument to the present case.

Section~\ref{sec:noregFin} uses the ``closeness in measure" statement of Section \ref{sec:noregMeas} and the approximation of blow up maps by homogeneous degree 1 harmonic maps described in Section~ \ref{sec:approx} to show the local product structure of a  harmonic maps $u$ near order 1 points.

 Section~\ref{sec:proofsofrigdity} uses the local product structure to first prove Theorem~\ref{maintechnical}.  Then the  rigidity results, Theorem~\ref{maintheorem} and Theorem~\ref{maintheorem'}, follow from adapting the classical Bochner techniques described in \cite{gromov-schoen}. 
 
 In Appendix \ref{app:proofof6}, we prove the technical results needed for Section \ref{sec:noregFin}. These results are contained in \cite{gromov-schoen}, \cite{sun}. However, we provide the complete proofs here as the results of this paper require that we better understand the constants, and what they depend on, that appear in their statements.

\subsection*{Acknowledgements} The authors are deeply grateful to A. Lytchak and Y.~Deng for their invaluable insights and enlightening discussions. \\

\tableofcontents


\section{Preliminaries}\label{sec:prelim}

\subsection{CAT(0) spaces}\label{CAT0sec}
A complete CAT(0) space $(X,d)$ generalizes the notion of a Hadamard manifold.  These are geodesic spaces of non-positive curvature where curvature is defined by triangle comparison; particular examples of CAT(0) spaces include Euclidean buildings (the focus of this paper) as well as Hadamard manifolds.  We refer to \cite{bridson-haefliger} for a complete introduction to these spaces.  {For  $P,Q \in X$ and $\lambda \in [0,1]$, we will denote the point on the unique geodesic segment $\overline{PQ}$ connecting $P$ and $Q$ at a distance   $(1-\lambda)d(P,Q)$ from $P$ by 
$
(1-\lambda)P + \lambda Q.
$
}

\subsection{Euclidean Buildings} Euclidean buildings are CAT(0) spaces with extra structure.
In this paper, we use Kleiner and Leeb's notion of Euclidean buildings and refer to \cite{kleiner-leeb} for the precise definition.    The equivalence of this notion and of that   by Tits~\cite{tits} was established by A.~Parreau \cite{parreau}. Here, we only highlight concepts important to this paper.    

Let $\mathbb E^N$ be an $N$-dimensional affine space and  $\partial_{Tits} \mathbb E^N\simeq \Sp^{N-1}$ be its Tits boundary.  Denote by  $\rho:\mathrm{Isom}(\mathbb E^N) \to \mathrm{Isom}(\partial_{Tits} \mathbb E^N)$ the canonical homomorphism which assigns to  each affine isometry its rotational part.  An affine Weyl group $W_{\mathrm{aff}}$ is a subgroup of $\mbox{Isom}(\mathbb E^N)$  generated by reflections and such that its {\bf reflection group} $W := \rho(W_{\mathrm{aff}}) \subset \mbox{Isom}(\partial_{Tits} \mathbb E^N)$ is finite.  The pair  $(\mathbb E^N, W_{\mathrm{aff}})$ is then called a {\bf Euclidean Coxeter complex}.  A {\bf wall} is a hyperplane of $\mathbb E^N$ which is the fixed point set of a reflection in $W_{\mathrm{aff}}$.

Let $(X,d)$ be a CAT(0) space endowed with a structure which makes it into a Euclidean building  modelled   on a Coxeter complex $(\mathbb E^N, W_{\mathrm{aff}})$ (in the sense of \cite[Section 4.1.2]{kleiner-leeb}). 
{We refer to such an $(X,d)$ (or more simply $X$) as a {\bf Euclidean building of type $W$.}
}
  We refer to the integer $N$ as the  {\bf dimension} of $X$.
As a consequence of \cite[Corollary 4.6.2]{kleiner-leeb},   there is a collection $\mathcal A$ of isometric embeddings $\iota : \R^N  \rightarrow X$ satisfying the following two properties: 
\begin{itemize}
\item Every geodesic segment, ray, and line is contained in an image of an isometric embedding of  the collection  (cf~\cite[EB3]{kleiner-leeb}), and 
\item Two isometric embeddings $\iota_1$ , $\iota_2$ of the  collection are compatible in the sense that $\iota_1^{-1} \circ \iota_2$ is a restriction of an isometry in $W_{\mathrm{aff}}$  (cf~\cite[EB4]{kleiner-leeb}) 
\end{itemize}
and such that $\mathcal A$ is the maximal collection satisfying the above two properties.
  We call $\mathcal A$ an {\bf atlas}, $\iota \in \mathcal A$ a {\bf chart} and the image   $A:=\iota(\R^N)$ an {\bf apartment}.

For those less familiar with buildings, we add the following lemma which is probably well-known to the experts.

\begin{lemma} \label{biggeraffinegp}
Let $(X,d)$ be as in the above paragraph.  
\begin{itemize}
\item If a Weyl group $W_{\mathrm{aff}}'$ contains $W_{\mathrm{aff}}$ as a subgroup, then  $(X,d)$ is  also endowed with a structure of a Euclidean building  modelled  on a Coxeter complex $(\mathbb E^N, W_{\mathrm{aff}}')$ with atlas $\mathcal A'=\{\iota \circ w:  \iota \in \mathcal A \mbox{ and } w \in W_{\mathrm{aff}}'\}$.
\item Let $\varphi:\R^N \rightarrow \R^N$ be an orthogonal transformation and define $$
\varphi \cdot W_{\mathrm{aff}} := \{\varphi^{-1} \circ f \circ \varphi:  f \in W_{\mathrm{aff}} \} \ 
\mbox{ and } 
\ \varphi \cdot \mathcal A:= \{\iota \circ \varphi:  \iota \in \mathcal A\}.
$$ 
Then $(X,d)$ is also endowed with a structure of a Euclidean building modelled on the Coxeter complex $(\mathbb E^N, \varphi \cdot W_{\mathrm{aff}})$ with atlas $\varphi \cdot \mathcal A$.
\end{itemize}
\end{lemma}

\begin{proof}
It is straightforward to check conditions EB1 $\sim$ EB4 of \cite[Section 4.1.2]{kleiner-leeb} and note that any two charts  for an apartment $A$ only differ by a precomposition with an isometry in $W_{\mathrm{aff}}$. 
\end{proof}

\begin{remark} \label{rem:varphiW}
 Let $\varphi:\R^N \rightarrow \R^N$ be an orthogonal transformation.
By Lemma~\ref{biggeraffinegp}, $(X,d)$ is a Euclidean building of type $\varphi \cdot W$ 
(also often described simply as a Euclidean building of type $W$).  
Note that  changing the building structure does not change $(X,d)$ as a CAT(0) space and does not affect the behavior of harmonic maps into $(X,d)$.  This fact is used in Section~\ref{sec:projectionsetup} and Section~\ref{subsec:inmeasure}.
\end{remark}

\subsection{Tangent cones} \label{sec:tangentcone}

Let $\tilde \angle_x(y,z)$  denote the angle at $x$ of the
comparison triangle in $\R^2$. If $y', z'$ are interior points on the segments $\overline{xy}$, $\overline{xz}$,
then $\tilde \angle_x (y',z') \leq \tilde \angle_x(y,z)$. Thus, $\lim_{y' \rightarrow x, z' \rightarrow x}\tilde \angle_x(y',z')$ exists, and we denote it by $\angle_x(y,z)$ (cf~\cite[Section 2.1.3]{kleiner-leeb}).
\begin{definition} \label{geodesicgerms}
Two geodesics $c_1, c_2$ emanating from a common point $x\in X$ are said to be equivalent if  $\angle_x(c_1, c_2)=0$.  A {\bf geodesic germ} at $x$ is an equivalence class of geodesics emanating from $x$.  The space of geodesic germs at $x$ along with the distance function defined by $\angle_x$ is a complete metric space by \cite[Lemma 4.2.2]{kleiner-leeb} and defines  the {\bf space of directions} $\Sigma_xX$.  By \cite[Section 4.2.2]{kleiner-leeb}, $\Sigma_xX$ is a spherical building modelled on $(\Sp^{N-1}, W)$. 
\end{definition}

\begin{definition} \label{def:tangentcone}
For $x \in X$, the {\bf tangent cone} $(T_x X,d_x)$ is a metric cone over $\Sigma_x X$.  Denote the vertex of $T_{x} X$ by  $\mathsf{O}$.  Any element of $T_{x} X \backslash \mathsf{O}$ can be written as
$([\gamma], t)$
where $[\gamma]$ is a geodesic germ at $x$ and $t \in (0,\infty)$. 
\end{definition}

\begin{lemma}
If $X$ is a Euclidean building of type $W$, then the tangent cone $(T_xX,d_x)$ is a Euclidean building of type $W$.
\end{lemma}

\begin{proof}
Let $\mbox{Isom}_0(\mathbb E^N)$ be the stabilizer subgroup of the origin and let $W_0 = \mbox{Isom}_0(\mathbb E^N) \cap \rho^{-1}(W)$.  
Using  \cite[Lemma 4.2.3]{kleiner-leeb}, it is straightforward to check that $T_xX$ is a Euclidean building modelled on $(\mathbb E^N, W_0)$ and thus $T_xX$ is a Euclidean building of type $W$.
 \end{proof}

If $A$ is an apartment of $X$ with $x \in A$, then $T_xA$ is an apartment of $T_xX$.  Conversely, every apartment of $T_xX$ is of the form $T_xA$ for an apartment $A$ in $X$ containing the point $x$ (cf~\cite[Lemma 4.2.3]{kleiner-leeb}).  

We remark that the tangent cone $T_xX$ is variously called the tangent cone (in \cite{kleiner-leeb}) or the geodesic cone $C_xX$ (in \cite{lytchak}, where a much more general notion of ``tangent cones" is defined).  We will see more notions from \cite{lytchak} in Section~\ref{sec:proj}.

\subsection{Harmonic maps}
The theory of harmonic maps into complete CAT(0) spaces was first developed in \cite{gromov-schoen} and generalized in \cite{korevaar-schoen1,korevaar-schoen2,jost}.  We refer the reader to these papers for a more complete introduction to harmonic maps into CAT(0) spaces.    In this paper, we are only interested in the case when the target space is a Euclidean building.

These papers also introduce the Sobolev space $W^{1,2}(\Omega,X)$ of finite energy maps from a Riemannian domain $(\Omega,g)$ to a CAT(0) space $X$.  For a map $u \in W^{1,2}(\Omega,X)$, they generalize  the energy density function $|\nabla u|^2$ and the directional energy density function $|u_*(V)|^2$ for a  Lipschitz vector field $V$ defined on $\Omega$. We reference \cite[Section 1]{korevaar-schoen1} for precise definitions.    Note that these functions depend on the  domain metric $g$, but we suppress $g$ from the notation of the energy density function whenever it is clear from the context which domain metric we are using.  Otherwise, we write $|\nabla u|_g^2$. Furthermore, we will write 
$$\left|\frac{\partial u}{\partial x_i}\right|^2
$$ to denote the directional energy density with respect to the coordinate vector field $\frac{\partial}{\partial x_i}$ and the domain metric $g$, although this is  denoted by $\left|u_*(\frac{\partial}{\partial x_i})\right|^2$ in \cite{korevaar-schoen1}.

\begin{definition}
For $u:(M,g) \to (X,d)$ and a fixed $\Omega \subset M$, the {\bf energy of $u$ on $\Omega$ in the metric $g$} is denoted
\[
{^g}E^u[\Omega]:=\int_\Omega|\nabla u|^2d\mu_g
\]where $\mu_g$ denotes the volume measure with respect to $g$. We let $\deltae$ denote the Euclidean metric and when $g = \deltae$ we remove the exponent and write $E^u[\Omega]$.
\end{definition}

To define harmonic maps, we use the notion of the trace of $u$, for $u\in W^{1,2}(\Omega,X)$, as defined in \cite{korevaar-schoen1}. We denote the space of admissible maps $W^{1,2}_u(\Omega, X):= \{ h \in W^{1,2}(\Omega, X) : d(u,h) \in W^{1,2}_0(\Omega)\}$. 

\begin{definition}\label{mindef} 
A  map $u:\Omega \rightarrow X$ is {\bf harmonic} if, for every $p \in \Omega$, there exists $r>0$ such that  the restriction $u|B_r(p)$  
minimizes energy amongst maps in
$W^{1,2}_u(B_r(p),X)$.
\end{definition} The existence and uniqueness of energy minimizers from Riemannian domains into CAT(0) spaces was established in \cite{korevaar-schoen1}.

\subsubsection{The order function}

Following \cite{gromov-schoen}, for a map $u\in W^{1,2}(\Omega , X)$, a fixed point $p \in \Omega$, and $\sigma >0$ such that $B_\sigma(p) \subset \Omega$,  we let 
\begin{align*}
E^u(\sigma):= \int_{B_\sigma(p)}|\nabla u|^2 \, d\mu_g \ \
\mbox{ and } \ \
I^u(\sigma):= \int_{\partial B_\sigma(p)}d^2(u, u(p))\, d\Sigma_g.
\end{align*}
We define the order of $u$ at $p$, at scale $\sigma>0$, by
\begin{equation}\label{eq:ordscale}
\text{Ord}^u(p,\sigma):=e^{c_1 \sigma^2} \frac{\sigma E^u(\sigma)}{I^u(\sigma)}
\end{equation}where $c_1$ depends on the $C^2$-estimates of the metric $g$. We are presuming here that $I^u(\sigma)$ is not zero, a fact which follows easily for harmonic maps by modifying the arguments of \cite[Section 2]{gromov-schoen} to all CAT(0) spaces. Following these arguments further implies that for appropriately chosen $c_1$, $\text{Ord}^u(p,\sigma)$ is monotone non-decreasing in $\sigma$ for all harmonic maps $u$. It therefore makes sense to define the order at a point.

\begin{definition}
Let $u:\Omega \to X$ be a harmonic map. Then the {\bf order of $u$ at $p \in \Omega$} is given by
\[
\text{Ord}^u(p):= \lim_{\sigma \to 0}\text{Ord}^u(p,\sigma).
\]
\end{definition}
Since $p \mapsto \text{Ord}^u(p, \sigma)$ is continuous for a fixed $\sigma$, the map $p \mapsto \text{Ord}^u(p)$ is upper semi-continuous.

\subsection{Tangent maps and blow up maps} \label{subsec:rescaling}
\label{ss:pullback}
We follow  \cite[Section 3]{gromov-schoen} to construct   homogeneous  maps associated to harmonic maps which we call {\bf tangent maps}.  Because we want to consider Euclidean buildings that are not necessarily locally compact, we also rely on the notion of {\it convergence in the pullback sense},  introduced in \cite{korevaar-schoen2}.    This is a generalization of the Arzela-Ascoli theorem for a sequence of maps  into the same target space.  Convergence in the pullback sense is defined for a sequence of maps where the target spaces  may be different and not necessarily compact.  {We summarize this notion in Section~\ref{sec:limitspace} below, but refer to  \cite[Section 3]{korevaar-schoen2} for more details.}

Consider a harmonic map $u:(\Omega,g) \to (X,d)$  and $p \in \Omega$. Choose local normal coordinates centered at $p$ and consider the restriction $u:(\B_r(0),g) \to (X,d)$. 
For $0<\sigma<1$, define
\begin{equation}\label{eq:usigmamap}
u_\sigma:(B_1(0),g_\sigma) \to (X, d_\sigma)
\end{equation}
where
\[
u_\sigma(x) := u\left(\sigma x\right), \ \  
g_\sigma(x) :=g\left(\sigma x\right),
\]
and
\[
d_\sigma(P,Q):=\left(\frac{\sigma^{n-1}}{I^u(\sigma)}\right)^{1/2}d(P,Q).
\]

 We will refer to $u_\sigma$ as the {\bf (Gromov-Schoen) blow up map}.  A computation involving  change of variables (cf~\cite[Section 3]{gromov-schoen}) shows that these maps have uniformly bounded  energy $E^{u_\sigma}(1) \leq 2\alpha$ where $\alpha =\text{Ord}^u(p)$ for $\sigma>0$ sufficiently small.  Thus, \cite[Theorem 2.4.6]{korevaar-schoen1} implies that these maps have uniform Lipschitz estimates in any compactly contained subsets of $B_1(0)$.  By \cite[Proposition 3.7]{korevaar-schoen2} and \cite[Theorem 3.11]{korevaar-schoen2}, for any sequence  $u_k=u_{\sigma_k}$, there exists a subsequence (which we will still denote by $u_k$) that converges locally uniformly in the pullback sense (cf~Section~\ref{sec:limitspace} below) to a limit  map $u_*:B_1(0) \rightarrow (X_*,d_*)$ into a CAT(0) space. 
 Furthermore, the energy density measures and the directional energy measures of $u_k$ converge weakly to those of $u_*$.  Moreover, following  \cite[last paragraph in the proof of Proposition 3.3]{gromov-schoen}, we see that $u_*$ is a {\it nonconstant homogeneous map of degree $\alpha$} (cf~Section~\ref{sec:homogeneous} below). The map  $u_*$ will be  referred to as a {\bf tangent map}.
The rescaling and the limit  preserve the order at the center. That is, $\alpha=\text{Ord}^u(p)=\text{Ord}^{u_\sigma}(0)= \text{Ord}^{u_*}(0)$.

\subsection{Convergence in the pullback sense}
\label{sec:limitspace}
We give a  brief summary of the limit space construction of  \cite[Section 3]{korevaar-schoen2} and explain the notion of convergence in the pullback sense.  

Following the notation of \cite[Section 3]{korevaar-schoen2}, let $\Omega_0=B_1(0)$ and  iteratively define 
$\Omega_{i+1}:=\Omega_i \times \Omega_i \times [0,1]$,  inclusion maps $\Omega_i \to \Omega_{i+1}$ by $x \mapsto (x,x,0)$  and $\Omega_\infty:= \bigcup \Omega_i$.
Next, let $d_\infty$ be a pseudodistance function defined on $\Omega_\infty \times \Omega_\infty$, denote its restriction to $\Omega_i \times \Omega_i$ by $d_i$, and assume
\begin{equation}
\label{eqn:pullback}
d_{i+1}^2 (z,(x,y,\lambda)) \leq (1-\lambda)d^2_{i+1}(z,(x,x,0))+ \lambda d_{i+1}^2(z,(y,y,0)) - \lambda(1-\lambda) d_i^2(x,y)
\end{equation}
for $x,y \in \Omega_i$, $z \in \Omega_{i+1}, \lambda \in [0,1]$.
Let $Z$ be 
  the metric completion of the quotient metric space 
  \newcommand{\bigslant}[2]{{\raisebox{.2em}{$#1$}\left/\raisebox{-.2em}{$#2$}\right.}}
$\bigslant{\Omega_\infty}{\sim}$   of  $\Omega_\infty$ where  $x \sim y$ if and only if   $d_\infty(x,y)=0$.  The assumption \eqref{eqn:pullback} implies that $Z$ is a CAT(0) space.

For each element  $u_k=u_{\sigma_k}$ of a sequence of blow up maps defined  in Section~\ref{subsec:rescaling}, let $u_{k,0}=u_k$ and iteratively define $u_{k,i+1}:\Omega_{i+1} \rightarrow X_k :=(X, d_k)$ from $u_{k,i}:\Omega_i \to X_k$ by setting
\[
u_{k,i+1}(x,y,\lambda) = (1-\lambda)u_i(x) + \lambda u_i(y)\text{ (cf Section \ref{CAT0sec})}.
\] 
Let $d_{k,i}$ be the pullback pseudodistance of the map $u_{k,i}$  defined on $\Omega_i \times \Omega_i$.   Then $d_{k,i+1}$ and $d_{k,i}$ satisfies (\ref{eqn:pullback}) by the CAT(0) triangle inequality in $X_k$. The  pseudodistance  $d_{k,\infty}$  on $\Omega_\infty \times \Omega_\infty$ is defined by  setting $d_{k,\infty}|\Omega_i \times \Omega_i=d_{k,i}$.
Similarly, define $u_{*,i}$, $d_{*,i}$, and $d_{*,\infty}$ starting from  $u_{*,0}=u_*$.
When we say $u_k$ {\bf converges locally uniformly in the pullback sense} to $u_*$, we mean that $d_{k,i}$ converges locally uniformly to the pullback pseudodistance $d_{*,i}$. 
  In this case, $X_*$ is (isometric to) the metric completion of the quotient metric space 
  $\bigslant{\Omega_\infty}{\sim}$   of  $\Omega_\infty$ where $x \sim y$ if and only if   $d_{*,\infty}(x,y)=0$. The map $u_*$  is the composition of the inclusion $\Omega \hookrightarrow \Omega_\infty$ followed by the  natural projection map of 
  $\Omega_\infty \to \bigslant{\Omega_\infty}{\sim} \subset X_*$.
  
\subsection{Ultralimits of metric spaces and maps}
\label{sec:ultra}
For details on  ultrafilters and ultralimits, we refer the reader to \cite[Section 2.4]{kleiner-leeb} and  \cite[Section 3.3]{lytchak}.  We only give a quick summary here:  
\begin{itemize}
\item Let  $(X_k,d_k,\star_k)$ be a sequence of  pointed metric spaces.    The ultralimit 
$$
(X_\omega,d_\omega):=\wlim (X_k,d_k,\star_k)
$$ is the quotient metric space of  the set of all sequences $(x_k)$ of points $x_k \in X_k$ with $\sup\{d_k(x_k, \star_k)\}<\infty$ with respect to  the pseudometric $\tilde d_\omega((x_k), (y_k)):= \wlim (d_k(x_k,y_k))$.  In other words, a point of $X_\omega$ is an equivalence class $[(x_k)]$ where $d_\omega([(x_k)],  [(y_k)])=\tilde d_\omega((x_k),(y_k))$.
\item Let $f_k: (\widehat X_k, \widehat d_k) \rightarrow (X_k,d_k)$ be a sequence of maps between metric spaces with a uniform local Lipschitz bound.  The ultralimit 
$$
f_\omega:=\wlim f_k: (\widehat X_\omega, \widehat d_\omega)  \to  (X_\omega, d_\omega)
$$  is the  locally Lipschitz map defined by $f_\omega(p):=[(f_k(p))]$.
\end{itemize}

Let $u_k:=u_{\sigma_k}:B_1(0) \rightarrow (X,d_k)$ be the sequence of blow up maps converging locally uniformly to $u_*:B_1(0) \to (X_*,d_*)$ as in Section~\ref{subsec:rescaling}.  Then there exists an isometric totally geodesic embedding 
$
\iota: X_* \to X_\omega 
$
such that  
\[
u_\omega =\iota \circ u_*.
\] 
Indeed, we construct $\iota$ by first defining $\hat \iota:\Omega_\infty \rightarrow X_\omega$ by setting $\hat \iota(x)=[(u_{k,i}(x))_{k=1}^{\infty}]$ for $x \in \Omega_i$.  Then 
$$ d_\omega(\hat \iota(x), \hat \iota(y)) = \wlim d_{k} (u_{k,i}(x), u_{k,i}(y)) = \lim_{k \rightarrow \infty} d_{k}(u_{k,i}(x), u_{k,i}(y)) =d_{*,i}(x,y).
$$
In particular, if $x\sim y$ then $\hat\iota(x)=\hat\iota(y)$.  Thus,   $\hat\iota$ descends to the quotient $\bigslant{\Omega_\infty}{\sim}$ and  can be  isometrically extended to its metric completion to define $\iota:X_* \to X_\omega$.  (See also \cite{kim}.)
\begin{remark}
\label{KLTheorem5.1}
Based on the above paragraph, we conclude the $u_k$  converges locally uniformly in the pullback sense to $u_\omega$.  {Thus, we can always replace $u_*$ by $u_\omega$.}  This fact is particularly useful because,  given a harmonic map $u$ to a Euclidean building $X$ of type $W$,   its tangent map  $u_\omega$ at $x \in \Omega$ maps into a Euclidean building $X_\omega=\wlim X_k$ of type $W$ (cf~\cite[Theorem 5.1.1]{kleiner-leeb}).
\end{remark}

\subsection{Homogeneous harmonic maps}
\label{sec:homogeneous}
As discussed in Section~\ref{ss:pullback},  a tangent map $u_*$ (and  hence $u_\omega$ of Remark~\ref{KLTheorem5.1}) is a homogeneous map.  In this section, we show that the image of a  homogeneous degree 1 map is contained in a single apartment.

\begin{definition}
Let $v:{B_r(0)}\subset \R^n \to (X,d)$ be such that $v \in W^{1,2}(B_{r}(0),X)$. We say $v$ is {\bf homogeneous degree $\alpha$}  if
$\text{Ord}^u(0,\sigma)=\alpha$ for all $\sigma \in (0,r)$.
\end{definition}

\begin{remark}
Since the proof of \cite[Lemma 3.2]{gromov-schoen} holds in any NPC space,   a homogeneous degree $\alpha$ map  $v:{B_r(0)}\subset \R^n \to (X,d)$ satisfies the following properties: For all $x \in \partial B_r(0)$ and $\lambda \in [0,1)$,
\begin{itemize}
\item
$
d(v(\lambda x), v(0)) = \lambda^\alpha \,d\left(v(x),v(0)\right)$.
\item
 The image of $\lambda \mapsto v(\lambda x)$ is a geodesic in $X$. 
 \end{itemize}
A map satisfying these two properties is referred to as  {\it intrinsically} homogeneous in \cite{gromov-schoen}.)
\end{remark}

A {\bf flat} (or more descriptively, an $m$-flat) $F$ of $X$ is an image of an isometric embedding $\iota_F:\R^m \rightarrow X$.   
A {\bf wall in $X$} is an image of a wall in $\R^N$ under an atlas.  A geodesic line is an example of a 1-flat and a wall is an example of an ($N$-1)-flat.
Every flat is contained in an apartment (cf~\cite[Proposition 4.6.1]{kleiner-leeb}), and $N$-flats are precisely the apartments of $X$ (cf~\cite[Corollary 4.6.2]{kleiner-leeb}).
The next proposition shows that the image of a homogeneous degree 1 harmonic map  into $X$ is, as expected, a flat.

\begin{proposition}\label{prop:flats}

If $L:\B_1(0) \to X$ is a homogeneous degree 1 harmonic map,  then there exists an $r_0 \in (0,1)$, an apartment $A$,  and an extension of $L|B_{r_0}(0)$ as a homogeneous degree 1 harmonic map $\hat L:\R^n \rightarrow A \subset X$.  In particular,    $F:=\hat L(\R^n)$ is a flat. 
\end{proposition}

\begin{proof}This proof roughly follows that of \cite[Theorem 3.1]{gromov-schoen}, but we need to account for the more pathological behavior of the non-locally compact target $X$.
Following the initial part of the proof of \cite[Theorem 3.1]{gromov-schoen}, we  deduce that there exists a linear map $v$ and an isometric and  totally geodesic embedding $J$ such that $L= J \circ v|B_1(0)$.  Note that this part of  their proof does not use their hypothesis that $X$ is locally compact.

If $X$ is a geodesic cone in Euclidean space (as they assume in  \cite[Theorem 3.1]{gromov-schoen}),  we can easily  extend the map $J$, currently defined only on $\mathscr V:=v(B_1(0)) \subset \R^m$ 
  to an isometric and totally geodesic embedding  defined on all of $\R^m$.  Since we are not assuming $X$ is a geodesic cone, some care must be taken to prove an analogous statement.  
  
  Let $p=L(0)$.    Since $T_{p}X$ is a cone, we can   extend $H:=\log_{p} \circ J:\mathscr V \rightarrow T_{p}X$ 
  to a map $\hat H:\R^m \rightarrow T_{p}X$.  Since $J$ is an isometric and totally geodesic  embedding of 
 $\mathscr V$,
  this means that 
  $H$, and hence  $\hat H$, defines an isometric and  totally geodesic embedding.  By \cite[Proposition 4.6.1]{kleiner-leeb} and \cite[Lemma 4.2.3]{kleiner-leeb}, the flat  $\hat H(\R^m)$ is contained in an apartment $T_{p}A$ of $T_{p}X$; ie~$\hat H:\R^m \to T_{p}A$.
Since $\log_{p}$ restricted to $A$ is an isometry with inverse $(\log_p|A)^{-1}:T_pA \to A$, we can lift $\hat H$ to define an isometric and totally geodesic map $\hat J=(\log_{p}|A)^{-1} \circ  \hat H:\R^m \to A$.  

\begin{equation*}
\begin{tikzcd}
    &  A \arrow[d, "\log_{p}" ] \\
  \R^m  \arrow[ur, "\hat J"] \arrow[r, "\hat H" ]
& T_{p}A \end{tikzcd}
\end{equation*}
Note that it is not necessarily true that $\hat J=J|\mathscr V$.
For example, if $L=J\circ v:(-1,1) \to X$ is a geodesic segment with $p=L(0)$, then there is a geodesic line $\mathscr L:\R \to X$  with $T_p(\mathscr L(\R)) = T_p(L(\R))$.  Since geodesic segments does not {\it uniquely} extend to a geodesic line, this does not imply $\mathscr L$ and $L$ agree on $(-1,1)$.  On the other hand,  $\angle_p(\mathscr L(t), L(t))=0$ for $|t| >0$ (cf~Section~\ref{sec:tangentcone} for the definition of $\angle_p$). By \cite[Lemma 4.1.2]{kleiner-leeb}, $\mathscr L$ and $L$ initially coincide; ie~there exists  $r_0>0$ such that $\mathscr L|(-r_0,r_0)= L|(-r_0,r_0)$ and $\mathscr L$ is an extension of $L|(-r_0,r_0)$.

For the general case, we will prove that $\hat J$ agrees with $J$ in a small neighborhood of the origin; that is, $\hat J$ is an extension of $J|v(B_{r_0}(0))$ for $r_0>0$ sufficiently small:
\begin{equation*}
\begin{tikzcd}
    &  A \cap B_\epsilon(p) \arrow[d, "\log_{p}" ] \\
 v(B_{r_0}(0))  \arrow[ur, "\hat J=J"] \arrow[r, "\hat H=H" ]
& T_pA \cap B_\epsilon(\mathsf O) \end{tikzcd}
\end{equation*}
Following the arguments of   \cite[Section 4.4]{kleiner-leeb} regarding Weyl cones in $X$, we pick finitely many points in $\Sigma_p(\hat J(\mathscr V)) = \Sigma_p(J(\mathscr V))$ whose convex hull is $\Sigma_p(\hat J(\mathscr V)) = \Sigma_p(J(\mathscr V))$.  Then the convex hull of the corresponding segments is the convex set $\hat J(\mathscr V) \cap J(\mathscr V)$ and is a neighborhood $\mathscr U$ of $p$ in $\hat J(\mathscr V)$ and in $J(\mathscr V)$. 
Choosing $r_0>0$ such that  $v(B_{r_0}(0)) \subset \mathscr U$, we have that $\hat J|v(B_{r_0}(0))=J|v(B_{r_0}(0))$.  In other words, $\hat L=\hat J \circ v:\R^n \to X$ is a homogeneous degree 1 harmonic extension of $L|B_{r_0}(0)=J \circ v|B_{r_0}(0)$.
  \end{proof}

\subsection{$(X,A,L)$-triples}
Many of our arguments use sequences of homogeneous degree 1 harmonic maps $L_k: \R^n \to (X_k, d_k)$ with different target spaces.  We introduce a notion which relates these maps to a fixed homogeneous degree 1 harmonic map. \begin{definition}\label{def:triple}
Let $X$ be a Euclidean building of type $W$ with the dimension of $X$ at least $2$, let $A$ be an apartment of $X$ and let $L:\R^n\to X$ be a homogeneous degree 1 harmonic map with $F:=L(\R^n) \subseteq A$.

Now consider another Euclidean building $X'$ of type $W$, an apartment $A'$ of $X'$, a homogeneous degree 1 harmonic map $L':\R^n \to X'$  with $F' = L'(\R^n)\subseteq A'$. If there exists an isometry $\phi:A\to A'$ compatible with $W$ (in the sense that if $\iota_1:\R^N \rightarrow A$ and $\iota_2:\R^N \rightarrow A'$ are charts in $\mathcal A, \mathcal A'$ respectively, then $\iota_2^{-1}\circ\phi\circ\iota_1:\R^N\to\R^N$ is an isometry with rotational part in $W$) and $L'=\phi\circ L$, then we say that $(X',A',L')$ is an {\bf $(X,A,L)$-triple}.
\end{definition}

\begin{lemma}\label{lem:triplesenergy}
If $(X', A', L')$ is an $(X, A, L)$-triple then for every $\Omega \subset \B_1(0)$,
\begin{equation}
E^{L'}[\Omega] =E^{L}[\Omega].
\end{equation}And for all $i = 1, \dots, n$,
\begin{equation}
\int_\Omega \left| \frac{\partial L'}{\partial x_i}\right|^2 \, d\mu_0 = \int_\Omega \left| \frac{\partial L}{\partial x_i}\right|^2 \, d\mu_0 ,
\end{equation}where $\mu_0$ is the $n$-dimensional Lebesque measure.

\end{lemma}
\begin{proof}
From the definitions, $d'(L'(x), L'(y)) = d'(\phi \circ L(x), \phi \circ L(y)) = d(L(x), L(y))$ and thus the result follows immediately.
\end{proof}

\begin{remark} For $\sigma \in \R^+$ and $\sigma^{-1}X$ denoting $(X, \sigma^{-1}d)$, $(\sigma^{-1}X,\sigma^{-1}A,\sigma^{-1}L(\sigma x))$ is an $(X,A,L)$-triple.
\end{remark}

\subsection{Notations and conventions}
We denote the $C^2$ distance between two metrics $g,h$ by
\[
\|g-h\|_{C^2(B_1(0))}:= \max_{i,j, k, l=1, \dots, n} \sup_{B_1(0)}  \left( |g_{ij}-h_{ij}| +\left |\frac{\partial g_{ij}}{\partial x^k}   -\frac{\partial h_{ij}}{\partial x^k}   \right| + \left |\frac{\partial^2 g_{ij}}{\partial x^k \partial x^l}   -\frac{\partial^2 h_{ij}}{\partial x^k \partial x^l}   \right| \right).
\]

It will often be convenient to work in the Euclidean coordinate system on a single normal coordinate chart of a Riemannian manifold $M$.  Recall that if $\sigma$ is sufficiently small, $(g_\sigma)_{ij}$ and $(\deltae)_{ij}$ are close in the sense that the $C^2$ norm
\[
\|g_\sigma-\deltae\|_{C^2(B_1(0))}\to 0 \text{ as } \sigma \to 0.
\]

We will let 
\[
\mu_g \text{ denote the volume measure with respect to }g
\]and\[
\mu_0^k \text{ denote the }k-\text{dimensional Lebesque measure}
\]and will suppress the $k$ when $k = n$.


\section{Projection into the Sub-building Defined by a Flat}
\label{sec:proj}
In this section, we investigate the closest point projection map $\pi:X \to \PF$ from a building to a subbuilding defined by a flat $F$. The goal is to  quantify the ``loss of energy"  when we compose a harmonic map with $\pi$ (cf~Proposition~\ref{anglek}).  In other words, we show that if $u:B_1(0) \to X$ is a harmonic map and $u(x) \notin \PF$, then  $|\nabla(\pi \circ u)|^2(x)$ is less than  $|\nabla u|^2(x)$  by a controlled amount.

Throughout this section, $(X,d)$ is a Euclidean building of type $W$ with dimension  at least $2$.   An $m$-flat is a copy of Euclidean space $\R^m$ isometrically and totally geodesically embedded in $X$.  We  fix an $m$-flat $F$. 

\subsection{Subbuilding defined by a flat}
A flat $F' \subset X$ is {\bf parallel to $F$} if the Hausdorff distance between $F$ and $F'$ is bounded.  Let $\PF$ be the union of all flats parallel to the flat $F$. 

\begin{lemma}\cite[Proposition 4.8.1]{kleiner-leeb} \label{Xnot}
$\PF$ is a convex subbuilding  and splits isometrically as 
\begin{equation} \label{ftimesy}
\PF \simeq \FY
\end{equation}
where  $Y$ is itself a Euclidean building.
\end{lemma}

Since $P_F$ is convex, we can define the closest point projection map
$$\pi:X \rightarrow \PF.$$

\begin{remark}\label{rem:thick}
In what follows, we need to maintain the building structure of $X$ when we consider the subbuilding $\PF$. That is, rather than consider the canonical building structure (where $\PF$ is thick \cite[Proposition 4.9.2]{kleiner-leeb}), we continue to view $\PF$ as a building of type $W$ and preserve all walls of $X$ which are contained in $\PF$. 
\end{remark}

Let $x \in X$ and $x_0 := \pi(x) \in \PF$.  
We write  $x_0=(f,y_0)$ and 
$$
T_{x_0}\PF \simeq \R^m  \times T_{y_0}Y
$$ 
using decomposition (\ref{ftimesy}).  {This product structure is important in our analysis of the energy of the composition $\pi \circ u$ of the projection with a harmonic map.  In particular, the product structure distinguishes the walls of $T_{x_0}\PF$ into two categories.
}   

\begin{definition} \label{categories}
A wall $T_{x_0}H$ of  $T_{x_0}\PF$  
{\bf contains $T_{x_0}F$} if $T_{x_0}H \simeq\R^m \times T_{y_0}H_Y$ where $T_{y_0}H_Y$ is a wall of $T_{y_0}Y$.  Otherwise, we say that $T_{x_0}H$  {\bf does not contain $T_{x_0}F$}. 
(In view of Remark \ref{rem:thick}, not every wall of $T_{x_0}\PF$ must contain $T_{x_0}F$.)

\end{definition}

{The main idea in the proof of the  ``loss of energy" of $\pi \circ u$ is that, if $x_0:=\pi \circ u(p) \neq u(p)$, then the image of the blow up map of $\pi \circ u$ is contained in a wall that is transverse to $T_{x_0}F$ (a wall that does not contain $T_{x_0}F$) where $F$ is the image of the homogeneous approximation $L$ of $u$ at $p$.   Since the energy density of $u$ at $p$ agrees with the energy density of $L$, this implies the loss of energy at $p$.  The goal of this section is to make this idea precise.}

\subsection{Differentiable maps}
 \label{sec:differentials}
 In this section, we consider the differentials of $u$ and $\pi\circ u$ in the sense of Lytchak \cite{lytchak} and Korevaar-Schoen \cite{korevaar-schoen1}. 
   
Given a sequence $\ee$ of positive numbers converging to zero and $x \in X$, we follow the notation of \cite{lytchak} and let
\begin{equation}\label{eq:ee}
(X_x^\ee, d_x^\ee):=\wlim\left(X,\frac{1}{\epsilon_j}d, x\right).
\end{equation}
Thus, a  point $Q=[(x_j)] \in X_x^\ee$ is an equivalence class  of a sequence of points $(x_j)$ in $X$ with $\frac{1}{\epsilon_j}d(x_j ,x)<C$ for some $C>0$.  Two sequences $(x_j)$ and $(y_j)$ are equivalent if the pseudo-distance $d_x^\ee((x_j), (y_j)):=\omega\mbox{-}\lim \frac{1}{\epsilon_j} d(x_j, y_j)=0$.

Let $\mathcal O$ denote the vertex point of $T_{x_0}X$. Recall that $T_{x_0}X$ is isometrically embedded in $X_{x_0}^{(\epsilon_j)}$ as follows:  For $Q_0 =([\gamma_0],t_0) \in T_{x_0}X \setminus \{\mathcal O\}$, we view $Q_0$ as an element of $X_{x_0}^{\epsilon_j}$ by
 \[
([\gamma_0],t_0) \mapsto [(\gamma_0(\epsilon_j t_0)].
 \] Here $[\gamma_0]$ denotes the equivalence class of a space of directions emanating from $x_0$. That is, given two arclength parameterizations $\gamma_1, \gamma_2$ of geodesics emanating from $x_0$, 
$$\gamma_1 \sim \gamma_2 \Leftrightarrow \angle_{x_0}(\gamma_1,\gamma_2)=0.$$  
Given a sequence $\epsilon_j \to 0$, the  blow-up of a Lipschitz map $f:\Omega \to X$ at $p$  with $x_0=f(p)$ is
 $$f_p^{(\epsilon_j)}:T_p\Omega \to X_{x_0}^{(\epsilon_j)}$$  
defined by
$$
f_p^{(\epsilon_j)}(v) = [(f(\exp_p(\epsilon_j v))] \in X_{x_0}^{(\epsilon_j)}.
$$  
\begin{definition}We say $f:\Omega \to X$ is {\bf (Lytchak) differentiable at $p$} if, for $f(p)=x_0$, 
\[
f_p^{(\epsilon_j)}(v) :=[(f(\exp_p(\epsilon_j v))] \in T_{x_0}X  \text{ for all } v \in T_p \Omega,
 \]
and $f_p^{(\epsilon_j)}(v)$  is independent of the choice of sequence  $\epsilon_j \to 0$ (cf~\cite[Definition 7.1]{lytchak}).  In this case, we denote the differential of $f$ at $p$ by 
 $$D_pf:T_p\Omega \to T_{x_0}X.$$
 \end{definition}
 Thus if,
$ f_p^{(\epsilon_j)}(v)=Q_0=([\gamma_0],t_0)$ then by definition 
\begin{equation} \label{epj}
\lim_{\epsilon_j \to 0} \frac{d(f(\exp_p(\epsilon_j v)), \gamma_0(\epsilon_j t_0))}{\epsilon_j} =0,
 \end{equation} 
and if $f$ is differentiable at $p$ then the independence in $\epsilon_j \to 0$ implies further that
\begin{equation} \label{tau}
\lim_{\tau \to 0} \frac{d(f(\exp_p(\tau v)), \gamma_0(\tau t_0))}{\tau} = 0.
\end{equation}

Let $u: (\B_1(0),g) \rightarrow (X,d)$ be a harmonic map.  By \cite[Theorem 2.4.6]{korevaar-schoen1}, $u$ is locally Lipschitz continuous.  Thus, \cite[Theorem 1.6]{lytchak} implies that the set
\[
U' :=\{p \in \B_1(0): D_pu \mbox{ and } D_p (\pi \circ u) \mbox{ exist}\}
\]  
is of full measure.

Since $u\in W^{1,2}(B_1(0),X)$ and $\pi \circ u \in W^{1,2}(B_1(0),\PF)$, there are associated generalized pullback metrics   by \cite[Theorem 2.3.2]{korevaar-schoen1}.  More precisely (replacing $\pi$ by $\mathfrak p$ in \cite[Theorem 2.3.2]{korevaar-schoen1} since $\pi$ is already used here for the projection map),  
\[
\mathfrak p^v: \Gamma(TB_1(0)) \times \Gamma(TB_1(0)) \rightarrow L^1(B_1(0))
\]
defined by 
\[
\mathfrak p^v(Z,W) =\frac{1}{4} |v_*(Z+W)|^2-\frac{1}{4} |v_*(Z-W)|^2, 
\]
is  symmetric, bilinear, non-negative, and tensorial for $v=u$ and $v= \pi \circ u$.  For $p \in B_1(0)$,  let $\{\partial_i^p\}$ be the coordinate vector fields  with respect to normal coordinates centered at $p$.   Define
\begin{align*}
U:=&\{p \in U': \mathfrak p_p^u(\partial_i^p, \partial_i^p) = d^2_{x}(D_pu(\partial_i^p), \mathcal O_{x})\\
&\text{ and } \mathfrak p_p^{\pi \circ u}(\partial_i^p, \partial_i^p) = d^2_{x_0}(D_p(\pi \circ u)(\partial_i^p), \mathcal O_{x_0}), \, \forall i=1,\dots, n\}
\end{align*}
where $x = u(p)$, $x_0=\pi \circ u(p)$ and $d_x$, $d_{x_0}$ and $\mathcal O_x$, $\mathcal O_{x_0}$ denote  the distance functions and the origins of $T_xX$, $T_{x_0}\PF$, respectively.  By \cite[Lemma 1.9.5]{korevaar-schoen1}, $U$ is full measure in $U'$ and thus in $B_1(0)$.

The  inner product structure defined by $\mathfrak p^{\pi \circ u}$ implies that the map $D_p(\pi \circ u)$ has the same pullback distance function as a linear map $\ell:\R^n \rightarrow \R^n$.\footnote{The matrix $A=(\mathfrak p^{\pi \circ u}_{ij})$ is a symmetric matrix and  has a decomposition $A=Q^t D Q= (\sqrt{D}Q)^t \sqrt{D} Q$ where $Q$ is an orthogonal matrix, $D$ is a diagonal matrix with eigenvalues on the diagonal entries, $\sqrt{D}$ is the diagonal matrix with square root of the eigenvalues on the diagonal entries, and $t$ means transpose.  The matrix $\sqrt{D}Q$ defines a linear map $\ell:\R^n \rightarrow \R^n$ such the pullback metric $\ell^*\delta$ is $\mathfrak p^{\pi \circ u}$.} Thus, by the same argument as Proposition~\ref{prop:flats},  the image of $D_p(\pi \circ u)$ is a flat of $T_{x_0}\PF$.

  \subsection{Differential of the projection into $P_F$}
We choose a point $p \in U$, so in particular both $u$ and  $\pi \circ u$  are (Lytchak) differentiable at $p$. Denote the differentials of $u$ and $\pi \circ u$ at $p$ as
\[
D_pu: T_pB_1(0) \to T_xX \ \mbox{ and } \ D_p(\pi \circ u): T_pB_1(0) \simeq \mathbb R^n \to T_{x_0}P_F.
\]
Since $p \in U$, by the discussion in the previous subsection, the differentials $D_pu$ and $D_p(\pi\circ u)$ can be viewed as linear maps from  $\mathbb{R}^n$.  Consequently, 
\[
D_p(\pi \circ u)(\R^n):=F_{x_0}
\] 
 is a flat and contained in some apartment  of $T_{x_0}P_F$.
Similarly, $F_x:=D_p u(\mathbb R^n) \subset T_xX$ is a flat of $T_xX$.

Let  $$\Pi:T_{x_0}X \to T_{x_0}P_F$$ be the closest point projection map.  
We will prove (cf~Lemma~\ref{Q} below) that
\[
F_{x_0}
\subset \Pi(T_{x_0}X).
\]

\begin{lemma} \label{boundarylemma}
Denote by  $\partial T_{x_0}P_F$  the topological boundary of $T_{x_0}P_F$ as a subset of  $T_{x_0} X$.  Then
 \[
F_{x_0}
\subset \partial T_{x_0}P_F.
 \]
In particular, 
for any $Q_0 \in F_{x_0}$, 
\[
T_{Q_0}(T_{x_0}X) \setminus T_{Q_0}(T_{x_0}P_F) \neq \emptyset.
\]
\end{lemma}

\begin{proof}
For an arbitrarily chosen  $Q_0 \in F_{x_0}$, we will show that $Q_0 \in \partial T_{x_0}P_F$.     If $Q_0=O_{T_{x_0}X}$, the vertex of the cone $T_{x_0}X$, then let $Q_\epsilon=([\sigma_0],\epsilon)$ where $\sigma_0(s)$ is the arclength parameterization of a geodesic from $x_0$ to $x$.  Then $Q_\epsilon \notin T_{x_0}P_F$ and $Q_\epsilon \to Q_0$ which proves $Q_0 \in \partial T_{x_0}P_F$.  

We will henceforth assume $Q_0 \neq O_{T_{x_0}X}$.
Let $v \in \mathbb R^n \simeq T_pB_1(0)$ be such that
\begin{eqnarray*}
D_pu(v) & = &  Q \ = \ ([\gamma],t) \in T_xX  \\ 
D_p(\pi \circ u)(v) & = &  Q_0 \ = \ ([\gamma_0],t_0) \in T_{x_0}P_F
\end{eqnarray*}
where $\gamma$ (resp.~$\gamma_0$) is an arclength parameterization of a  geodesic ray  with $\gamma(0)=x$ (resp.~$\gamma_0(0)=x_0$). 

The assumption that  $u$ (resp.~$\pi \circ u$) is Lytchak differentiable at $p$  implies that 
$$
\lim_{\tau \to 0^+} \frac{d(u(\exp_p(\tau v)), \gamma(t\tau))}{\tau} =0
$$
and 
$$
\lim_{\tau \to 0^+} \frac{d(\pi \circ u(\exp_p(\tau v)), \gamma_0(t_0\tau))}{\tau} =0.
$$

Since $\pi$ is 1-Lipschitz, the above equalities together with the triangle inequality imply
\begin{equation} \label{wd'''}
\lim_{\tau \to 0^+} \frac{d(\pi \circ \gamma(t\tau), \gamma_0(t_0\tau))}{\tau} =0.
\end{equation}
Among other things, \eqref{wd'''} implies that $\lim_{\tau \to 0} \angle_{x_0}(\overline{x_0\, \pi \circ \gamma(t \tau}), \gamma_0)=0$.

  Let $\sigma_0(s)$ be the arclength parameterization of the geodesic from $x_0$ to $x$. 
\begin{claim}
\begin{equation} \label{ang}
\angle_{x_0}(\gamma_0, \sigma_0)=\pi/2.
\end{equation} 
\end{claim}
\begin{proof}
 Assume for the purposes of a contradiction that $\angle_{x_0}(\sigma_0,\gamma_0)>\frac\pi2$; and in particular there exists $\epsilon>0$ such that $\angle_{x_0}(\sigma_0,\gamma_0)=\frac\pi2+3\epsilon$. Let $p_\tau:= \gamma(\tau t)$.
By the continuity of angles \cite[Proposition 3.3]{bridson-haefliger}, for sufficiently small $\tau$ we have that $\angle_{x_0}(\overline{x_0p_\tau},\gamma_0^+)>\frac\pi2+2\epsilon$.  Moreover, by \eqref{wd'''} we have, for sufficiently small $\tau$, that $\angle_{x_0}(\gamma_0,\overline{x_0\pi(p_\tau)})<\epsilon$.  The triangle inequality for angles then implies that for such $\tau$,
\begin{equation}\label{eq:large-npc-angle}
\angle_{x_0}(\overline{x_0p_\tau},\overline{x_0\pi(p_\tau)})>\frac\pi2+\epsilon.
\end{equation}

Now, consider the geodesic triangle $\Delta(p_\tau,\pi(p_\tau),x_0)$ and the corresponding Euclidean comparison triangle $\Delta_E(A,B,C)$.  Because Euclidean comparison angles are greater than their CAT(0) counterparts, \eqref{eq:large-npc-angle} implies that
\begin{equation}\label{eq:large-comp-angle}
\angle_{C}(\overline{CA},\overline{CB})>\frac\pi2+\epsilon.
\end{equation}
In particular, this means that the side opposite $C$ (the side $\overline{AB}$ corresponding to $\overline{p_\tau \pi(p_\tau)}$ in the original) is the {\em longest} side of the Euclidean comparison triangle. In particular its length is greater than the length of $\overline{AC}$ (which is equal to the length $\overline{x_0p_\tau}$ in the original).  Hence, we have that
\[
d(p_\tau,\pi(p_\tau))>d(p_\tau,x_0)
\]
which contradicts the minimality of the projection.
\end{proof}

Let $\alpha$ be an apartment in $T_{x_0}X$ containing $Q_0=([\gamma_0], t_0)$ and $Q_1=([\sigma_0], 1)$ and $A$ be an apartment of $X$ such that $T_{x_0}A=\alpha$.   Then there exists  $\tau_1, s_0>0$ such that  $\sigma_0(s) \in A$ for $s \in[0,s_0]$ and $\gamma_0(\tau) \in A$ for $\tau \in [0,\tau_1]$. Identify $A$ with $\mathbb R^n$ with standard coordinates denoted by  $(x^1, \dots, x^n)$ such that  $x_0$ identified with the origin $(0,\dots, 0)$ of $\mathbb R^n$.  

By \eqref{ang}, we can assume that $\gamma_0$ (resp.~$\sigma_0$) parameterizes the initial segment of the positive $x^1$-axis (resp.~$x^2$-axis) in $\mathbb R^n$.
Let 
\[
c_\epsilon:[0,\infty) \to \mathbb R^n \simeq A, \ \ c_\epsilon(\tau) = (1/\sqrt{1+\epsilon^2}, \epsilon/\sqrt{1+\epsilon^2}, 0, \dots, 0).
\]
be an parameterization of a  ray in $A$. 
Define
$Q_\epsilon:=([c_\epsilon], t_0).$ Then
\[
Q_\epsilon \to Q_0 \mbox{ as $\epsilon \to 0$  in }T_{x_0}X.
\]
We are done by proving
\[
Q_\epsilon \notin T_{x_0}P_F, \mbox{ for $\epsilon>0$ sufficiently small.}
\]
To prove this statement, it is sufficient to show that $c_{\epsilon}( \tau) \notin P_F$ for $\tau\in [0,\tau_\epsilon]$.  On the contrary, assume that $c_{\epsilon}( \tau) \in P_F$ for some $\tau_\star >0$.  By the convexity of $P_F$ and the fact that $c_\epsilon(0)=x_0 \in P_F$, this also implies that $c_{\epsilon}( \tau) \in P_F$ for all $\tau \in [0, \tau_\star]$. Fix $\tau_1>0$ small. By (\ref{ang}), $\angle_{x_0}(c_\epsilon,\sigma_0)<\pi/2$.  Thus,   the closest point to $\sigma_0(\tau_1)$ on the ray $c_\epsilon([0,\infty))$ is not $x_0$. This contradicts $\pi(\sigma_0(\tau_1))=\pi(x)=x_0$, thereby proving that $c_\epsilon( \tau) \notin P_F$.   
\end{proof}

\begin{lemma} \label{PiC}
Let $Q_0 \in F_{x_0}$ and let $O_{Q_0}$ be the vertex of the cone $T:=T_{Q_0}(T_{x_0}X)$ and let  $C:=T_{Q_0}(T_{x_0}P_F) \subset T$.  Denote by $\Pi_C:T \to C$ the closest point projection map.  Then there exists $w_0 \in T \setminus C$ such that $\Pi_C(w_0)=O_{Q_0}$. 
\end{lemma}

\begin{proof}
By Lemma~\ref{boundarylemma}, there exists  $z_1 \in T \setminus C$.   Let $c_1 :=\Pi_C(z_1)$ and denote
\[
z_t=(1-t)O_{Q_0} + tz_1 \ \ \mbox{ and } \ \ c_t=(1-t)O_{Q_0}+tc_1.
\]

\vspace*{0.2in}

\noindent {\sc Claim 1}. $\Pi_C(z_t)=c_t$ for all $t \in (0,1)$. \\

{\sc Proof of claim.} Assume on the contrary that there exists $t \in (0,1)$ such that $\Pi_C(z_t)=p \neq c_t$.  Since $T$ is a cone, there is a  unique geodesic ray   from $O_{Q_0}$ going through $p$.  Let $\bar p \in T$ be the point on this geodesic ray with $t(O_{Q_0},\bar p)=d(O_{Q_0},p)$.  Since distances in the cone scale linearly along radial geodesics, we have 
\[
td(z_1,\bar p)=d(z_t,p).
\]
Since $p \in C$ and $C$ is a cone, $\bar p \in C$.  Thus, the fact that  $\Pi_C(z_t)=p \neq c_t$ and $c_1=\Pi_C(z_1)$ implies
\[
d(z_t,p)<d(z_t,c_t)=td(z_1,c_1)\leq td(z_1,\bar p)=d(z_t,p).
\] 
The above inequality is strict since $p\neq c_t$. This contradiction  proves the claim. \hfill $\Box$(claim)\\

Since $T_{x_0}X$ is a Euclidean building, its tangent cone at any point is again a Euclidean building, and any two points lie in a common apartment.  Thus, we can choose an apartment $A_1$  of $T$  containing $z_1$ and $c_1$.  
Let $t \mapsto w_t$ be the unique geodesic line in $A_1$ with  $w_1=z_1$ and parallel to the geodesic $t \mapsto c_t$.
  In other words, $\overline{c_tw_t}$ is a geodesic extension of $\overline{c_tz_t}$ for each $t \in (0,1)$.\\
  \\
{\sc Claim 2.} $\Pi_C(w_t)=c_t$ for $t \in (0,1]$.  \\

{\sc Proof of Claim.} On the contrary, assume $\Pi_C(w_t)=\gamma_t \neq c_t$ for $t \in (0,1]$.  Since $\Pi_C(z_t)=c_t$, there exists a unique point $w_\star \in \overline{w_t z_t} \cap \overline {w_t \gamma_t}$ such that $\angle_{w_\star}({\overline{w_\star \gamma_t}, \overline {w_\star z_t}})>0$. Note that $\Pi_C(w_\star)=\gamma_t$ so in particular, while $w_\star$ might equal $w_t$, $w_\star \neq z_t$. Then the geodesic quadrilateral defined by $c_t, \gamma_t, w_\star, z_t$ has sum of interior angles $> 2\pi$.  Indeed, 
$\angle_{c_t}(\overline{c_t\gamma_t},\overline{c_tz_t})$,  $\angle_{\gamma_t}(\overline{\gamma_tc_t},\overline{\gamma_tw_\star}) \ge \pi/2$ by the characterization of closest-point projections in CAT(0) spaces, $\angle_{z_t}(\overline{z_tc_t}, \overline{z_tw_\star})=\pi$ since $z_t$ is a point on the geodesic $\overline{c_tw_\star}$, and $\angle_{w_\star}(\overline{w_\star z_t}, \overline{w_\star\gamma_t})>0$.  This is a contradiction since  quadrilaterals in CAT(0) spaces has sum of interior angles $\leq 2\pi$.  
\hfill $\Box$(claim)
\\
\\
We are now in position to finish the proof of the lemma. Indeed, by the continuity of $\Pi_C$, we have
\[
O_{Q_0}=\lim_{t \to 0} c_t=\lim_{t \to 0} \Pi_C(w_t) = \Pi_C (\lim_{t \to 0} w_t) = \Pi_C(w_0).
\]
\end{proof}

\begin{lemma} \label{Q}
For $Q_0 \in F_{x_0}$, there exists $\hat{Q} \in T_{x_0}X$ such that $\Pi(\hat{Q})=Q_0$.  In other words, $F_{x_0} \subset \Pi(T_{x_0}X).$
\end{lemma}

\begin{proof}
For $w_0 \in T_{Q_0}(T_{x_0}X) \setminus T_{Q_0}(T_{x_0}P_F)$ of Lemma~\ref{PiC},  write $w_0=([\overline{Q_0\hat{Q}}], t_0)$.  The fact that  $\Pi_C(w_0)=O_{Q_0}$ implies  $\angle_{Q_0}(\overline{Q_0\hat{Q}}, \overline{Q_0Q'}) \geq \pi/2$ for any $Q' \in T_{x_0}P_F$. By the characterization of closest-point projections in CAT(0) spaces, this implies $\Pi(\hat{Q})=Q_0$.
\end{proof}

We now prove two technical lemmas about projection maps into conical Euclidean buildings.  Lemma \ref{betterwall} is crucial for establishing the energy loss, as it provides a description of the image of the blow-up of $\pi \circ u$ at $p$ (cf~Proposition \ref{lemma:star}). 
 
 Let  $B$ be a cone over a spherical building with vertex $\mathsf{O}$ and let $C$ be a subbuilding. Furthermore, suppose that $C$ is the set of all flats in $B$ parallel to a given flat $\mathfrak f$. That is, 
\[
C=P_{\mathfrak f} \simeq \mathfrak f \times C', \ \ \ \mathfrak f \mbox{ is a flat}.
\]
As in Remark \ref{rem:thick}, we preserve the walls of $B$ in $C$ which the canonical building structure would remove.  We have the following analog of Definition~\ref{categories}.

\begin{definition}
Let $\mathfrak h$ be a wall of $C$.  
We say {\bf  $\mathfrak h$  contains $\mathfrak f$} if it is of the form $\mathfrak f \times \mathfrak h'$ where $\mathfrak h'$ is a wall of $C'$. Otherwise, we say $\mathfrak h$ {\bf does not contain} $\mathfrak f$.
\end{definition}

\begin{lemma} \label{goodwall}
 If $C \neq B$, then there exists a wall $\mathfrak h$ of $C$ that does not contain $\mathfrak f$ and $\mathsf{O} \in \mathfrak h$.  
\end{lemma}

\begin{proof}
Let  $\mathfrak a_0$ be an apartment of $C$ and $\mathfrak a_1$ be an apartment of $B$ such that $\mathfrak a_1 \not \subset C$.  
By \cite[Corollary 4.4.6]{kleiner-leeb}, 
$$\mathfrak a_0 \cap \mathfrak a_1=\bigcap_{i=1}^I \mathfrak a_{0i}^+
$$ 
where $\mathfrak a_{01}^+, \dots, \mathfrak a_{0I}^+$ are half-apartments of $\mathfrak a_0$.  If  $\mathfrak f \subset \mathfrak a_{0i}^+$ for all $i=1, \dots, I$, then $\mathfrak f \subset \mathfrak a_1$ which would imply that $\mathfrak a_1$ is an apartment of $C$, a contradiction.  Thus,  there exists a half-apartment $\mathfrak a_{0i}^+$  such that $\mathfrak f \not \subset \mathfrak a_{0i}^+$ which implies that the wall $\mathfrak h := \partial \mathfrak a_{0i}^+$ does not contain $\mathfrak f$.  Lastly, $\mathsf{O} \in \mathfrak h$ since $C$ is a cone.
\end{proof}

\begin{lemma}\label{betterwall}
Let $\pi_C:B \to C$ be the closest point projection map.
If $C\neq B$, $Q\in B\setminus C$, and $Q_0=\pi_C(Q)$, then there exists a wall $\mathfrak h$ of $C$ that does not contain $\mathfrak f$ such that $Q_0\in\mathfrak h$. 
\end{lemma}
\begin{proof}
Let $\widetilde{C}=T_{Q_0}C$, $\widetilde{B}=T_{Q_0}B$, and $\widetilde{\mathfrak f}=T_{Q_0}\mathfrak f$.  We observe that the geodesic germ $\overline{QQ_0}$ cannot lie in $T_{Q_0}C$ (because $Q_0$ is the closest point of $C$ to $Q$), and therefore $\widetilde{C}\neq\widetilde{B}$.  Applying Lemma~\ref{goodwall}, we conclude that there is a wall $\widetilde{\mathfrak h}$ of $\widetilde{C}$ which does not contain $\widetilde{\mathfrak f}$  and $\mathsf O_{Q_0} \in \widetilde{\mathfrak h}$ where $\mathsf O_{Q_0}$ is the vertex of the tangent cone $T_{Q_0}C$.  

It follows from \cite[Lemma 4.2.3]{kleiner-leeb} and the Euclidean building structure for tangent cones that $\widetilde{\mathfrak h}$ is a wall of $\widetilde C$ if and only if $\widetilde{\mathfrak h}=T_{Q_0}\mathfrak h$ for a wall $\mathfrak h$ of $C$ containing $Q_0$.  Hence, there is a wall $\mathfrak h$ of $C$ that does not contain $\mathfrak f$ with $Q_0 \in \mathfrak h$.
\end{proof}

Using the above lemma, we demonstrate that the full image of the differential map $D_p(\pi \circ u)$ is contained in a single wall which does not contain $T_{x_0}F$.
\begin{proposition} \label{lemma:star}
Let $p \in U$, $x=u(p)$ and $x_0=\pi \circ u(p)$. 
If $x \neq x_0$, there exists a   wall $T_{x_0}H$ that does not contain $T_{x_0}F$ 
such that $F_{x_0}$, the image of $D_p(\pi \circ u)$, is contained in $T_{x_0}H$.
\end{proposition}

\begin{proof}
Let $Q_0 \in D_p(\pi \circ u)(\R^n)$. By Lemma \ref{Q}, there exists a $\hat Q \in T_{x_0}X$ such that $\Pi(\hat Q) = Q_0$ (and $\hat Q \neq Q_0$). 

Set $B=T_{x_0}X$ and $C=T_{x_0}\PF$. By \cite[Corollary 3.6]{ahl},  $T_{x_0}\PF = P_{T_{x_0}F}$ inside $T_{x_0}X$,  so $C$ is the collection of all flats parallel to $\mathfrak f =T_{x_0}F \subset T_{x_0}X$. Applying Lemma \ref{betterwall} with this $B$ and $C$, there exists a wall $T_{x_0}H_{Q_0}$ containing $Q_0$ but not containing $T_{x_0}F$.
Now, since the image of $D_p(\pi \circ u)$ is a flat, $D_p(\pi \circ u)(\R^n)$ lies in a single apartment $T_{x_0}A \subset T_{x_0}P_F$. If there did not exist a single wall containing the full image of $D_p(\pi \circ u)$, then the linear sum of those subspaces would be all of $T_{x_0}A$. The proposition follows.\end{proof}

\subsection{Loss of energy} 
\label{sec:projectionsetup}
Finally, we demonstrate that at points $p$ where $u, \pi\circ u$ are differentiable and $u(p) \neq \pi \circ u(p)$, projecting into $F$ results in a quantitative energy loss (cf~Proposition~\ref{anglek}). Before we give a proof of the loss of energy we prove a technical result crucial for establishing the loss of energy statement. Specifically, we demonstrate an upper bound on a component function of the differential of $\pi_F$
(cf~Lemma~\ref{angle} below).

For $x_0=(f,y_0) \in \PF$, we denote the projection to the first component  by 
\[
\pi_F:\PF \simeq F \times Y \to F, \ \ \ \pi_F(x_0)=\pi_F(f,y_0)=f.
\]
Fixing an identification $F \simeq \R^m$, we can also view $\pi_F$ as a $\R^m$-valued map
\[
\pi_F:\PF \to \R^m.
\]
We denote the   $i^{th}$-component function by 
$$\pi_F^i:\PF \to \R.
$$

Let $A$ be an apartment of $\PF$.  Then $A\simeq F \times Y_A$ where $Y_A$ is an apartment of $Y$.   Let $\iota:\R^N \rightarrow A$ be a chart.  Let $\varphi:\R^N \rightarrow \R^N$ be an orthogonal transformation such that 
\begin{equation} \label{otchoice}
\begin{array}{l}
\iota \circ \varphi (\R^m \times \{(0, \dots, 0)\}) =F \simeq \R^m \mbox{ is the identity map of $\R^m$,}\\
\iota \circ \varphi(\{(0, \dots, 0) \times \R^{N-m}\})=Y_A.
\end{array}
\end{equation}
Then $\iota_A:=\iota \circ \varphi$ is a chart of the building $X$ of type $\varphi \cdot W$ (cf~Remark~\ref{rem:varphiW}).  
The chart $\iota_A$ induces the natural identifications
\begin{eqnarray} \label{tangentspace}
T_fF \simeq \R^m\simeq \R^m \times \{(0,\dots,0)\} \subset  \R^N \simeq   T_{x_0}A, & 
T_{x_0} \PF \simeq \R^m \times T_{y_0}Y.
\end{eqnarray}
With the above identifications, the restriction 
$
\pi_F|A \rightarrow F
$
 is simply the projection to the first $m$-components of $\R^N$; ie
 \begin{equation} \label{eq:proj}
 (x^1, \dots, x^m, x^{m+1}, \dots, x^N) \mapsto (x^1, \dots, x^m, 0, \dots, 0).
 \end{equation}
Furthermore, the (classical) differential $d(\pi_F|A): T_{x_0}A \simeq \R^N \rightarrow T_{\pi_F(x_0)}F \simeq \R^m$ of a projection map $\pi_F|A$  is of course also given by  (\ref{eq:proj}).
Similarly,  the restriction $\pi_F^i|A$ and its (classical)  differential $d(\pi_F^i|A)$ are both given by 
 \begin{equation} \label{eq:projith}
 (x^1, \dots, x^m, x^{m+1}, \dots, x^N) \mapsto  x^i.
 \end{equation}

The following lemma shows (\ref{eq:proj}) and (\ref{eq:projith}) also describe the restriction to $T_{x_0}A$ of the blow up maps 
\begin{eqnarray*}
(\pi_F)_{x_0}^{(\epsilon_j)}:{\PF}_{ x_0}^\ee \rightarrow (\R^m)_{\pi_F(x_0)}^\ee \simeq \R^m,
\\
(\pi_F^i)_{x_0}^{(\epsilon_j)}:{\PF}_{ x_0}^\ee \rightarrow \R_{\pi_F^i(x_0)}^\ee \simeq \R \ \ \ \ \ \ \ \,
\end{eqnarray*} 
 by proving that blowing up and taking the restriction commute; ie~$(\pi_F)_{x_0}^\ee |T_{x_0}A= d(\pi_F|A)$ and $(\pi_F^i)_{x_0}^\ee|T_{x_0}A = d(\pi_F^i|A)$.
\begin{lemma} \label{lem:rest}
With the identification $T_{x_0}A \simeq \R^N$ of (\ref{tangentspace}), the restriction $(\pi_F)_{x_0}^{(\epsilon_j)}\big|T_{x_0}A$ is given by (\ref{eq:proj}).  Similarly, the restriction of $(\pi_F^i)_{x_0}^{(\epsilon_j)}\big|T_{x_0}A$ is given by (\ref{eq:projith}).
\end{lemma}

\begin{proof}
For $([\gamma],t)\in T_{x_0}A$,   the geodesic germ $[\gamma]$ is represented by a geodesic contained in $A$ (cf~\cite[Lemma 4.1.2]{kleiner-leeb}).  View $([\gamma],t)$ as a point in ${\PF}_{ x_0}^\ee$ under the identification defined by the exponential map, $([\gamma],t)=[(\gamma(t \epsilon_j))]$.  Since $(\gamma(t \epsilon_j))$ is a sequence of points in $A$, 
$$(\pi_F)_{x_0}^{(\epsilon_j)} ([\gamma],t) = [(\pi_F(\gamma(t \epsilon_j)))]=[(\pi_F|A(\gamma(t \epsilon_j)))]= (\pi_F|A)_{x_0}^{(\epsilon_j)}  ([\gamma],t).
$$ Since Lytchak's notion of differentials (cf~\cite{lytchak}) agrees with the classical differentials on $\R^N$, we conclude $(\pi_F)_{x_0}^{(\epsilon_j)}|T_{x_0}A= (\pi_F|A)_{x_0}^\ee= d_{x_0}(\pi_F|A)$. 
The second assertion follows from a similar argument.
\end{proof}

Let  $T_{x_0}H$ be a wall  of $T_{x_0}\PF$.  Choose an apartment $T_{x_0}A$ of $T_{x_0}\PF$ containing $T_{x_0}H$ and  use the identification (\ref{tangentspace}) to define
\[
\theta(H,i) :=\min_{Q \in T_{x_0}H \backslash \{{\mathcal O}_{x_0}\}} \angle_{{\mathcal O}_{x_0}}(Q, \vec e_i).
\]
where  $\mathcal O_{x_0}$ denotes the origin of the tangent cone $T_{x_0}\PF$ and $\vec e_1, \dots, \vec e_m$ are the standard basis of $\R^m \simeq F$.  We then have the following:
\begin{itemize}
\item For a wall $T_{x_0}H$ that contains $T_{x_0}F$, we have $\theta(H,i)=0$ for all $i \in \{1, \dots, m\}$.  
\item For a wall $T_{x_0}H$ that does not contain $T_{x_0}F$,
\[
0 < \Theta(H) \leq \frac{\pi}{2}  \ \mbox{ where } \Theta(H):= \max_{i=1, \dots, m}  \theta(H,i),
\]
\end{itemize}
\begin{remark}\label{rem:thetanot}
Although the quantity $\Theta(H)$ was defined by choosing an apartment $T_{x_0}A$ that contains $T_{x_0}H$ and the identification (\ref{tangentspace}) induced by the chart $\iota_A:\R^n \to A$, it  does not depend on these choices.  The choice of a chart  $\iota_A$ resulted from the choice of the orthogonal transformation $\varphi$ satisfying (\ref{otchoice}).  Any other choice of an orthogonal transformation which also satisfies \eqref{otchoice} does not change the angle used to define $\theta(H,i)$.  Furthermore,  let $T_{x_0}A'$ be another apartment that contains $T_{x_0}H$ and $\iota':\R^N \rightarrow A'$ be a chart in $\varphi \cdot \mathcal A$. This implies that $\iota'^{-1} \circ \iota_A:\iota_A^{-1}(A \cap A') \to \iota'^{-1}(A \cap A')$ is a restriction of $w' \in \varphi \cdot W_{\mathrm{aff}}$.  Thus, $\iota_{A'}:=\iota' \circ w'$ is a chart such that 
 \begin{equation}\label{eq:chartchoice}
\iota_{A'}^{-1} \circ \iota_A|\iota_A^{-1}(A \cap A') \to \iota_{A'}^{-1}(A \cap A')\text{ is the identity map.}
\end{equation}
In particular, for each apartment $A'$ such that $T_{x_0}H$ is contained in $T_{x_0}A'$, there exists a unique chart satisfying \eqref{eq:chartchoice} that induces the same identification as (\ref{tangentspace}) with $A$ replaced by $A'$.
  \end{remark}

\begin{lemma} \label{angle}
There exists $\theta_0 \in (0,\frac{\pi}{2}]$ with the following property:  For $x_0 \in \PF$ and   a wall  $T_{x_0}H$ that does not contain $T_{x_0}F$,   there exists $i_H\in \{1,\dots, m\}$ such that
$$
 |(\pi_{F}^{i_H})_{x_0}^\ee (Q) |^2 \leq \cos^2 \theta_0 \cdot d^2_{ x_0}(Q,\mathcal O_{x_0}), \ \ \forall Q \in T_{x_0}H.
$$
Here,  $\mathcal O_{x_0}$ denotes the origin of the tangent cone $T_{x_0}\PF$ and $d_{x_0}$ denotes the distance function on $T_{x_0}\PF$.
\end{lemma}

\begin{proof} 
 Since  the walls through any point $x_0 \in \PF$ are determined by the finite reflection group $\varphi \cdot W$, there are only a finite number of possible values for $\theta(H,i)$.  Thus, there exists  $\theta_0 \in (0,\frac\pi2]$, which can be chosen independently of  $x_0 \in \PF$, such that
\[
\theta_0 \leq \Theta(H) \leq \frac{\pi}{2} \ \mbox{ for any wall $T_{x_0}H$ that does not contain $T_{x_0}F$.}
\]
We emphasize that $\theta_0$ is determined only by $F$ and $W$.
Choosing an apartment $T_{x_0}A$ containing $T_{x_0}H$, and using the identification (\ref{tangentspace}), for $i_H:=\mathrm{argmax}_{i=1, \dots, m} \theta(H,i)$, 
$$
 |\text{proj}_{\vec e_{i_H}}Q|=|Q \cdot \vec e_{i_H}| \leq  |\cos \Theta(H)| |Q|  \leq \cos \theta_0 \cdot |Q|, \ \ \ \forall Q \in T_{x_0}H \subset \R^N.
 $$
Thus, the result  follows from Lemma~\ref{lem:rest}.
\end{proof}

Finally, we can demonstrate the quantitative energy loss.

\begin{proposition} \label{anglek}
Let $u^i:= \pi_F^i \circ \pi \circ u:  B_1(0) \rightarrow \R$.  
There exists $\theta_0 \in (0,\frac{\pi}{2}]$ with the following property:  For every $p \in U$  such that $u(p) \notin \PF$,  there exists  $i \in \{1,\dots, m\}$ such that
\[
 \left|\frac{\partial u^i}{\partial x^i} \right|^2(p) \leq \cos^2 \theta_0 \cdot \left|\frac{\partial u}{\partial x^i} \right|^2(p).
 \]
\end{proposition}
\begin{proof}
Choose normal coordinates centered at $p \in U$ and let $x=u(p)$, $x_0=\pi \circ u(p)$.
  By Proposition \ref{lemma:star}, we may choose a wall $T_{x_0}H$ that does not contain $T_{x_0}F$ such that  $D_p(\pi \circ u)(\R^n) \subset T_{x_0}H$. 
  Let $\theta_0$ and $i=i_H$ be from Lemma~\ref{angle}.  
By the chain rule for blown up maps, we have
\[
D_pu^i = (\pi_{F}^i)_{x_0}^\ee \circ D_p(\pi \circ u).
\]
Thus, Lemma~\ref{angle} implies
  \[
 \left|\frac{\partial u^i}{\partial x^i} \right|^2(p) = \left|(\pi_{F}^i)_{x_0}^\ee \circ D_p(\pi \circ u)\left(\frac{\partial}{\partial x^i}\right) \right|^2  \leq \cos^2 \theta_0 \cdot d_{x_0}^2\left(D_p(\pi \circ u)\left(\frac{\partial}{\partial x^i}\right), \mathcal O_{x_0}\right).
\]
Since  the projection map $\pi:X \rightarrow \PF$ is distance non-increasing, 
\[
d_{x_0}^2\left(D_p(\pi \circ u)\left(\frac{\partial}{\partial x^i}\right), \mathcal O_{x_0}\right)
\leq d_x^2\left(D_pu \left(\frac{\partial}{\partial x^i}\right), \mathcal O_x\right)= \left|\frac{\partial u}{\partial x^i} \right|^2(p).
\]
The desired inequality follows from combining the above two inequalities.
\end{proof}


\section{Closeness in Measure}
\label{sec:noregMeas}

In this section, we establish the cornerstone of the proof for Theorem~\ref{maintechnical}. The statement asserts that if a harmonic map into $X$ is sufficiently close to a homogeneous degree 1 harmonic map, then the image of the harmonic map mostly lies in a subbuilding defined by the homogeneous degree 1 map.  The precise statement we need is contained in Proposition~\ref{Sun Lemma 12 Collapsed}. In the proof of the regularity theorem of Gromov-Schoen (cf~\cite[Theorem 5.1]{gromov-schoen}), this assertion follows from the fact that, for the locally finite case, a  homogeneous degree 1 map $L$  is {\it effectively contained} in the subbuilding $P_F$ where $F=L(\R^n)$.

However, we cannot expect this property to be true in a general Euclidean building.   Instead, we take advantage of the observation in the previous section that a projection into $P_F$ results in a  loss of energy.  We use this to prove the ``closeness in measure" assertion of  Proposition~\ref{Sun Lemma 12 Collapsed}.

\begin{proposition}\label{Sun Lemma 12 Collapsed} 
Fix $E_0>0$,  $r_0 \in (0,1)$ and a  homogeneous degree 1 harmonic map 
\[
L:\R^n \rightarrow A \subset X
\]
where $X$ is of type $W$, the dimension of $X$ is at least $2$, and $A$ is an apartment of $X$.
For every $\epsilon>0$,  there exists $\eta=\eta(\epsilon, E_0,r_0,X,A,L)>0$ such that the following holds:\\
\\
Let $(X',A',L')$ be an $(X,A,L)$-triple and let $u:(B_1(0),g)\to (X',d)$ be a harmonic map with ${^gE}^u[B_1(0)]\leq E_0$. 
 If  $\sup_{B_{r_0}(0)} d(u,L') < \eta$ 
 and $\|g-\deltae\|_{C^2}<\eta$,
then
\[
\mu_0\{p \in B_{r_0}(0):u(p) \notin P_{F'}\} < \epsilon
\]
where $P_{F'}$ denotes the union of all flats of $X'$ parallel to $F':=L'(\R^n)$.
\end{proposition}

 We will prove Proposition~\ref{Sun Lemma 12 Collapsed} by contradiction.  Therefore, we  assume that there exists $\epsilon>0$, a sequence of $(X,A,L)$-triples $(X_k, A_k, L_k)$, and a sequence of harmonic maps $u_k:(B_1(0), g_k) \to (X_k, d_k)$ satisfying
\begin{itemize}
\item ${^{g_k}E}^{u_k}[B_1(0)]\leq E_0$, 
\item $\sup_{B_{r_0}(0)} d(u_k,L_k) < \frac{1}{k}$, 
\item $\|g_k-\deltae\|_{C^2}<\frac{1}{k}$
\end{itemize}
and such that 
\begin{equation} \label{geqep}
\mu_0\{q \in B_{r_0}(0):u_k(q) \notin \PFk \} > \epsilon
\end{equation}
where $\PFk$ is the union of all flats of $X_k$ parallel to  $F_k:=L_k(\R^n) \subset A_k$.
Before we finish the proof of Proposition~\ref{Sun Lemma 12 Collapsed} (cf~Section~\ref{subsec:SunCollapsed}), we prove some preliminary lemmas (cf~Lemma~\ref{upsilonprime} and Lemma~\ref{omega*}) regarding the convergence in measure of the   directional energies of the harmonic maps $u_k$.

\subsection{Convergence in measure} \label{subsec:inmeasure}
Using an appropriate chart to identify $A \simeq \R^N$, we view $L$ as a linear map $L:\R^n \to \R^N$.  Furthermore, by the singular value decomposition,   after an appropriate orthogonal change of coordinates $\psi:\R^n \rightarrow \R^n$ and $\varphi:\R^N \rightarrow \R^N$, we can express $\varphi \circ L\circ \psi: \R^n \rightarrow \R^N$ by an  $(N \times n)$-matrix 
\[
\begin{bmatrix}
\Lambda &{\bf O}  \\
{\bf O} & {\bf O}
\end{bmatrix}.
\]
Here,  $\Lambda$ is a  diagonal $(m\times m)$-square matrix with diagonal entries $\lambda_1 \geq \dots \geq  \lambda_m>0$ and the ${\bf O}'s$ represents zero matrices (of appropriate sizes).  In particular, 
\[
F=L(\R^n) \simeq \R^m \simeq \R^m \times \{(0,\dots,0)\} \subset \R^N.
\]
We will proceed using these new coordinates on $\R^n$ and $\R^N$ and changing the building structure of $(X,d)$ to the one of type $\varphi \cdot W$ (cf~Remark~\ref{rem:varphiW}).  By an abuse of notation, we denote $\varphi \circ L\circ \psi$ again by $L$.

Recall from Definition \ref{def:triple} that there exists an isometry $\phi_k:A \to A_k$, compatible with the building structure, such that $L_k = \phi_k \circ L$. 
In particular, 
$
F_k =L_k(\R^n) =\phi_k \circ L(\R^n).
$
We use $\phi_k$ to identify 
\begin{equation} \label{simeq}
F_k  \simeq  F \simeq \R^m.
\end{equation} 
Under this identification  via $\phi_k$,  
\begin{equation}\label{star}
L_k \equiv L.
\end{equation} 
In particular, the $i^{th}$-coordinate function $L^i_k$ of $L_k$ and the $i^{th}$-coordinate function $L^i$ of $L$ are the same function, and we have 
\begin{eqnarray*}
 \frac{\partial L^i_k}{\partial x^i}  \ = \  \frac{\partial L^i}{\partial x^i} & = &  \lambda_i \ \mbox{ for } i=1, \dots, m,  
\\ 
 \frac{\partial L^j_k}{\partial x^i}  \ = \  \frac{\partial L^j}{\partial x^i} & = &  0 \ \mbox{ for }i \neq j \mbox{ or } i=j=m+1, \dots, n. \nonumber
\end{eqnarray*}
In particular,  for $i=1, \dots, m$, 
\begin{equation}  \label{lambdai}
\left|  \frac{\partial L_k}{\partial x^i} \right|^2 = \left|  \frac{\partial L}{\partial x^i} \right|^2= \lambda_i^2.
\end{equation}

\begin{lemma} \label{upsilonprime}
Fix $\tau >0$ and $i=1, \dots, m$.  
\[
\mbox{If }
\Upsilon_k(\tau, i):=\left\{ p \in  B_{r_0}(0):   (1+ \tau)^2 \lambda_i^2  \leq  \left|\frac{\partial u_k}{\partial x^i} \right|^2(p)\right\},
\mbox{ 
then }
\lim_{k \to \infty}\mu_0(\Upsilon_k(\tau, i))=0.
\]
\end{lemma}
\begin{proof}
Let $\epsilon_0>0$ be given.
We invoke \cite[last paragraph of the proof of Theorem 2.4.6]{korevaar-schoen1} which states that $\left|\frac{\partial u_k}{\partial x^i} \right|^2$ is essentially subharmonic (cf~Remark~\ref{re:thm2.4.6} below).  By a standard argument, we obtain a mean value inequality of the form
\begin{equation} \label{almostMVI}
\left|\frac{\partial u_k}{\partial x^i} \right|^2(p)  \leq \frac{c}{\mu_{g_k}(B_r(p))}\int_{B_r(p)}  \left|\frac{\partial u_k}{\partial x^i} \right|^2 d\mu_{g_k}
\end{equation}
where the constant $c=c(r,g_k)$ depends on the $C^2$-closeness of $g_k$ to the Euclidean metric in $B_r(p)$.  More precisely,  $c=1$ if the domain metric is Euclidean, and  we can choose $r>0$ sufficiently close to $0$ and $K_0 \in \N$ sufficiently large such that 
\[
c <(1+\tau), \ \ \forall k  \geq K_0.
\]
Assume further that $r>0$ is so small that $\mu_0(B_{r_0}(0)\backslash B_{r_0-r}(0))<\epsilon_0/2$. Then for all $p \in B_{r_0-r}(0)$, $B_r(p) \subset B_{r_0}(0)$. Define 
\[
\bar u^r_k:B_{r_0-r}(0) \rightarrow \R, \ \ \ \bar u^r_k(p) = (\mu_{g_k}(B_r(p)))^{-1}\int_{B_r(p)}  \left|\frac{\partial u_k}{\partial x^i} \right|^2 d\mu_{g_k}.
\]

It is straightforward to show that the sequence of measures $\left|\frac{\partial u_k}{\partial x^i} \right|^2 d\mu_{g_k}$ converges weakly to  $\lambda_i^2 \, d\mu_0$ in $B_{r_0}(0)$.\footnote{We first prove that $u_k$ converges uniformly in the pullback sense to $L$ in $B_{r_0}(0)$. To do so, following the notation of Section~\ref{sec:limitspace}, we  need to  check inductively in $j$ that the sequence of the pullback pseudodistance function of $u_{k,j}$ converges uniformly to the pullback pseudodistance function of   $L_j$ on $\Omega_j \times \Omega_j$, where $\Omega=B_{r_0}(0)$.  This follows easily because: (i) The pullback pseudodistance of  $L_k$  is equal to the pullback pseudodistance of $L$ since $L_k=\phi_k \circ L$ for an isometry $\phi_k:A\to A_k$. (ii) The difference between pullback pseudodistance of $u_k$ and  the pullback pseudodistance of  $L_k$ converges to 0 since $\sup_{B_{r_0}(0)}d_k(u_k, L_k) < 1/k$. 
Next, the convergence of measures follows from (\ref{lambdai}) and  \cite[Theorem 3.11]{korevaar-schoen2}.}
The Portmanteau Theorem implies (since $B_r(p)$ is a continuity set for the measure $\lambda_i^2d\mu_0$ normalized to a probability measure) that
\[
\lim_{k \rightarrow \infty} \int_{B_r(p)}   \left|\frac{\partial u_k}{\partial x^i} \right|^2 d\mu_{g_k}
=
\lambda_i^2 \mu_0(B_r(p)).
\]
Since $\mu_{g_k}(B_r(p)) \rightarrow \mu_0(B_r(p))$ as $k \rightarrow \infty$, we conclude that $\bar u^r_k \rightarrow \lambda_i^2$ pointwise in  $B_{r_0-r}(0)$. By Egorov's theorem, there exists a set $V$ with  $\mu_0(V)< \epsilon_0/2$ and $K \geq K_0$ such that $|\bar u^r_k(p) -\lambda_i^2|<\tau \lambda_i^2$ for all $k \geq K$ and $p \in B_{r_0-r}(0) \backslash V$.   Thus, for all
 $k\geq K$ and  $p \in B_{r_0-r}(0) \backslash V$,  
 \[
\left|\frac{\partial u_k}{\partial x^i} \right|^2(p)  \leq c \bar u^r_k(p) \leq c\left(1+\tau \right) \lambda_i^2 \leq (1+\tau)^2 \lambda_i^2.
\]
In other words,  $\Upsilon_k(\tau, i) \cap B_{r_0-r}(0) \subset V$ for all $k\geq K$.  Therefore,
\[
\mu_0(\Upsilon_k(\tau, i)) \leq \mu_0(V) +\mu_0(B_{r_0}(0)\backslash B_{r_0-r}(0))<\epsilon_0, \ \ \ \forall k \geq K.
\]
\end{proof}

\begin{remark}  \label{re:thm2.4.6}
If the domain metric  is Euclidean, the weak subharmonicity of $\left|\frac{\partial u_k}{\partial x^i} \right|^2$ follows from  \cite[Remark 2.4.3]{korevaar-schoen1}.  Indeed, if $u:B_1(0) \rightarrow X$ is a harmonic map with respect to the Euclidean metric on $B_1(0)$, then
 $$
 \int_{B_1(0)} \nabla d^2 (u,u_{sw}) \cdot \nabla \eta \ d\mu_0\geq 0
 $$ 
 for a constant vector $w$, $s \in \R$ and $\eta \in C^\infty_c(B_1(0))$ where $u_{sw}(x)=u(x+sw)$.  We let $w=\frac{\partial}{\partial x^i}$ and divide by $s^2$ and let $s \rightarrow 0$ to prove $\left|\frac{\partial u}{\partial x^i} \right|^2$ is weakly subharmonic.  If the metric  $g$ on $B_1(0)$ is not Euclidean, we  follow the proof of \cite[Theorem 2.4.6]{korevaar-schoen1}.  We now set $u_{sw}(x)=u(\bar x(x,s))$ where  $\bar x(x,s)$ is the one-parameter family of flows defined by the vector $w$ and the metric $g$.  We observe that the constant $C$ that appears in the (subharmonicity) inequality \cite[(2.4.xxv)]{korevaar-schoen1}
 \[
 \int_{B_1(0)} |\nabla u |^2 (\triangle \eta + C |\nabla \eta| + C\eta) \, d\mu_g \geq 0
 \]
is due to the difference of the measures  $(g^{ij} d\mu_g)_{sw}$ and $(g^{ij} d\mu_g)_{-sw}$ to $g^{ij} d\mu_g$, and hence $C \rightarrow 0$ as $\|g -\deltae\|_{C^2} \rightarrow 0$.
The  standard  technique of letting $\eta$ be radially symmetric test functions in the above inequality proves the mean value inequality (\ref{almostMVI}) where $c$ depends on $C$ and hence on the metric $g$, and   $c \rightarrow 1$ as $\|g -\deltae\|_{C^2} \rightarrow 0$.\end{remark}

For each $k$,
 denote the closest point projection map from $X_k$ to $\PFk$ by
  \[
 \pi_k: X_k \rightarrow \PFk.
\]
Following the setup of Section~\ref{sec:projectionsetup}, denote the projection onto the first component of $\PFk \simeq F_k \times Y_k$ by
\[
\pi_{F_k}:  \PFk\rightarrow \R^m \simeq F_k
\]
and let $\pi_{F_k}^i$ be the $i^{th}$-component function of $\pi_{F_k}$. Finally, define
\[
u_k^i : = \pi^i_{F_k} \circ  \pi_k \circ u_k.
\]

\begin{lemma} \label{omega*}
Fix  $\delta>0$ and $i\in \{1, \dots, m\}$. 
\[
\mbox{If }
\Theta_k(\delta, i): =\left\{ p \in B_{r_0}(0):   \left|\frac{\partial u_k^i}{\partial x^i}(p) \right|^2 \leq  (1-\delta)^2\lambda_i^2 \right\}, 
\mbox{
then } 
\lim_{k\to\infty}\mu_0(\Theta_k(\delta, i))=0. 
\]
\end{lemma}

\begin{proof} 
We consider the case $i=1$ as all other cases follow similarly.    On the contrary, assume there exists a subsequence of $k \rightarrow \infty$ (which we will still denote as $k$ by an abuse of notation) such that $\lim_{k \to \infty}\mu_0(\Theta_k(\delta, 1)) \geq \beta >0$.

Let ${\bf B} \subset \R^{n-1}\simeq \{0\} \times \R^{n-1}$ be the ball of radius $r_0$ centered at the origin. For $p \in B_{r_0}(0)$, write $p = (p_1, \bar p)$ where $\bar p:= (p_2, \dots, p_n) \in {\bf B}$. 
Use the projection of $B_{r_0}(0) \rightarrow {\bf B}$, $(p_1,\bar p) \mapsto \bar p$ to view $B_{r_0}(0)$ as a fiber bundle over ${\bf B}$ with intervals as fibers.  
More precisely,  to each $\bar p \in {\bf B}$, we associate an interval  $I_{\bar p}:=(-\rho(\bar p), \rho(\bar p)) \subset \R$  where $\rho(\bar p)=\sqrt{(3/4)^2 - |\bar p|^2}$.

For $\bar p \in {\bf B}$, define a subset $\theta_k(\bar p)$ of the interval $I_{\bar p}$ by 
\[
\theta_k(\bar p):= \left\{p_1 \in I_{\bar p}:  \left| \frac{\partial u^1_k}{\partial x^1} (p_1,\bar p) \right|\leq (1-\delta)\lambda_1\right\}.
\]
 Define a subset   $\mathcal A_k$ of the base space ${\bf B}$ by 
$$\mathcal A_k:= \{\bar p \in {\bf B}: \mu_0^1(\theta_k(\bar p))> \beta/4\omega\}
$$ (recall  $\mu_0^k$ denotes the $k$-dimensional Lebesgue measure) where $\omega:=\mu_0^{n-1}({\bf B})$. 
Now suppose that $\liminf_{k \to \infty} \mathcal A_k =0$. Then there exists a subsequence (again labeled by $k$) such that 
 $\lim_{k \to \infty}\mu_0^{n-1}(\mathcal A_k^c) =\omega$. (We use superscript $c$ to denote the complement of a set.)
 
 Then Fubini's theorem implies
\begin{align*}
\frac{\beta}{4} &= \lim_{k \to \infty}\mu_0^{n-1}(\mathcal A_k^c) \cdot \frac{\beta}{4\omega} \geq \lim_{k \to \infty} \int_{\mathcal A_k^c} \mu_0^1(\theta_k(\bar p)) d\mu_0^{n-1}\\& =\lim_{k \to \infty} \int_{{\bf B}}\mu_0^1(\theta_k(\bar p)) d\mu_0^{n-1}= \lim_{k \to \infty}\mu_0(\Theta_k(\delta, 1)) \geq \beta,
\end{align*}
a contradiction.  Thus, $\liminf_{k \to \infty}\mu_0^{n-1}(\mathcal A_k) >0$. 

For $\kappa, \tau>0$ to be chosen later, define a subset $\mathcal B_k$ of the base ${\bf B}$ by
\begin{align*}
\mathcal B_k:= \left\{ \bar p \in {\bf B}:\right.& \left.\left| \frac{\partial u^1_k}{\partial x^1} (p_1, \bar p) \right|\leq (1+\tau) \lambda_1 \text{ for } p_1 \in I_{\bar p}\right.\left.  \text{ except on a subset of measure}< \kappa \right\}.
\end{align*} 
Lemma \ref{upsilonprime} implies that $\lim_{k \to \infty}\mu_0^{n-1}(\mathcal B_k) = \omega$. Thus,  there exists $K$  with the property that $\mathcal A_k \cap \mathcal B_k \neq \emptyset$ for all $k> K$.  Choose $\bar p_k \in \mathcal A_k \cap \mathcal B_k$ for each $k$.

Define a subset $b_k$ of the interval $I_{\bar p_{k}}$ by
\[
 b_k:=\left \{t \in I_{\bar p_k}: \left| \frac{\partial u^1_k}{\partial x^1} (t, \bar p_k) \right| \leq (1+\tau) \lambda_1 \right\}.
\]
For ease of notation, let $$
f_k(t):= u_k^1(t, \bar p_k).
$$
 The inequality $|f_k'(t)|^2 \leq \left|\frac{\partial u_k}{\partial  t}\right|^2(t, \bar p_k)$ and the uniform Lipschitz bounds on $u_k$ in $B_{r_0}(0)$ imply there exists $C>0$ such that 
 $$
 |f_k'(t)| \leq C, \ \ \forall t \in I_{\bar p_k}=(-\rho(\bar p_k), \rho(\bar p_k)).$$

Since $\bar p_k \in \mathcal A_k \cap  \mathcal B_k$,
\[
\mu_0^1(\theta_k(\bar p_k)) > \beta/4\omega \ \ \mbox{ and } \ \ 
\mu_0^1(b_k^c)<\kappa.
\]
Therefore, for $k >K$,
\begin{align*}
f_k(\rho(\bar p_k))- f_k(-\rho(\bar p_k))&=
 \int_{-\rho(\bar p_k)} ^{\rho(\bar p_k)}f_k'(t)\, dt
  \nonumber \\
& \leq \int_{-\rho(\bar p_k)} ^{\rho(\bar p_k)}|f_k'(t)|\, dt \nonumber \\
&=\int_{\theta_k(\bar p_k)} |f_k'(t)|\, dt+\int_{\theta_k(\bar p_k)^c \cap b_\kappa} |f_k'(t)|\, dt+\int_{\theta_k(\bar p_k)^c \cap b_\kappa^c} |f_k'(t)|\, dt 
\nonumber \\
& \leq\mu_0^1(\theta_k(\bar p_k))(1-\delta)\lambda_1+ \mu_0^1(\theta_k(\bar p_k)^c) (1+\tau) \lambda_1+ C \mu_0^1(b_\kappa^c)\nonumber \\
& =   2 \rho(\bar p_k)  \lambda_1
-\mu_0^1(\theta_k(\bar p_k))\delta \lambda_1+ \mu_0^1(\theta_k(\bar p_k)^c) \tau \lambda_1+ C \mu_0^1(b_\kappa^c) \nonumber
\\
&< 2 \rho(\bar p_k)  \lambda_1
-\frac{\beta \delta \lambda_1}{4\omega}+ \mu_0^1(\theta_k(\bar p_k)^c) \tau \lambda_1+ C \kappa. \nonumber
\end{align*}
Thus, by choosing $\kappa, \tau>0$ sufficiently small (depending only on $\beta, \delta, \lambda_1, \omega, C$), we conclude
\begin{equation} \label{rhocontra}
f_k(\rho(\bar p_k))- f_k(-\rho(\bar p_k)) <  2 \rho(\bar p_k)  \lambda_1
-\frac{\beta \delta \lambda_1}{8\omega}, \ \ \forall k >K.
\end{equation}

On the other hand, for any $p \in B_{r_0}(0)$, we can view $u_k^1(p)$, $L^1(p)$ as points in $\R \simeq  \R \times \{0,\dots, 0\} \subset \R^N \simeq A_k \subset \PFk$.  Under this identification,
\[
|L^1(p)- u_k^1(p)| = d_k(L^1(p),  u_k^1(p)).
\]
Since $\sup_{B_{r_0}(0)} d_k(L^1(p),u_k^1(p)) 
   \leq 
   \sup_{B_{r_0}(0)} d_k(L_k(p), u_k(p)) < \frac{1}{k}$,
   we conclude
\[
\lim_{k \rightarrow \infty}  \sup_{p \in B_{r_0}(0)}|L^1(p)- u_k^1(p)| = 0.
\]
In particular, this implies $f_k(\rho(\bar p_k))- f_k(-\rho(\bar p_k)) \to 2 \rho(\bar p_k)\lambda_1$ as $k \rightarrow \infty$, contradicting (\ref{rhocontra}).
\end{proof}

\subsection{Completion of the proof of Proposition~\ref{Sun Lemma 12 Collapsed}} \label{subsec:SunCollapsed}

 We are now in position to contradict inequality~(\ref{geqep}) and finish the proof.

We apply Proposition~\ref{anglek} with $u$ replaced by  $u_k$ and $U$ replaced by an analogous set $U_k$ defined by $u_k$.  Thus, there exists $\theta_0 \in (0,\frac{\pi}{2}]$ (independent of $k$) such that  
$$
\forall p \in U_k \cap \{q \in \B_{r_0}(0): u_k(q) \not \in \PFk\},$$ 
there exists  $i  \in \{1, \dots, m\}$ such that 
\[
 \left|\frac{\partial u_k^i}{\partial x^i} \right|^2(p) \leq \cos^2 \theta_0  \cdot  \left|\frac{\partial u_k}{\partial x^i} \right|^2(p).
\]
For $\tau>0$,  the above inequality implies
\[
 \left|\frac{\partial u_k^i}{\partial x^i} \right|^2 (p)\leq   \cos^2 \theta_0   (1+\tau)^2  \lambda_i^2 \mbox{ \ or \ } \ (1+\tau)^2 \lambda_i^2   \leq   \left|\frac{\partial u_k}{\partial x^i} \right|^2(p).
 \]
In other words, 
\[
U_k \cap \{q \in \B_{r_0}(0): u_k(q) \not \in \PFk \} \subset \bigcup_{i=1}^m \Theta_k (\delta, i) \cup \Upsilon_k(\tau, i).
\]
where $\tau>0$ satisfying $\cos^2 \theta_0 (1+\tau)^2<1$ defines $\Upsilon_k(\tau, i)$ (cf~Lemma~\ref{upsilonprime}) and   $\delta>0$  satisfying $(1-\delta)^2=\cos^2 \theta_0 (1+\tau)^2<1$ defines $\Theta_k(\delta, i)$ (cf~Lemma~\ref{omega*}). 
For a large enough $k$, Lemma~\ref{upsilonprime} and Lemma~\ref{omega*}  imply that
$$
\mu_0\left(U_k \cap \{q \in \B_{r_0}(0): u_k(q) \not \in \PFk\} \right)<\epsilon.
$$
Since $U_k$ is of full measure by \cite[Theorem 1.6]{lytchak}, we conclude 
 \[
 \mu_0\left( \{q \in B_{r_0}(0):u_k(q) \notin \PFk\}\right)<\epsilon
 \]
which  contradicts (\ref{geqep}) and completes the proof of Proposition~\ref{Sun Lemma 12 Collapsed}.

\section{Homogeneous Approximations}\label{sec:approx}
 In a locally finite Euclidean building, any point has a neighborhood that can be isometrically and totally geodesically embedded into the tangent cone at that point. Thus, one can assume that $u$ locally maps into a cone, and  the sequence of blow up maps (defined in Section~\ref{subsec:rescaling}) all have the {\it same} conical target space.  Invoking the Arzela-Ascoli theorem, a subsequence of blow up maps converges to a tangent map, also into the same conical target space. Hence one can approximate the sequence of  blow up maps (and hence $u$) by a  {\it single} homogeneous degree 1 map.

As we are not presuming that $X$ is locally finite, we cannot assume that $u$ maps into a cone.  Thus, the blow up maps and tangent maps have different target spaces. The goal of this section is to construct a {\it sequence} of homogeneous degree 1 maps corresponding to a sequence of blow up maps such that the target space of the corresponding maps agree.

The main result of this section is Proposition \ref{LinearApprox}.
For simplicity, we use the following notation:   For the Gromov-Schoen blow up maps $u_\sigma:(B_1(0),g_\sigma)\to (X, d_\sigma)$ defined by \eqref{eq:usigmamap} and a sequence $\sigma_k \rightarrow 0$, we let
\[
u_k=u_{\sigma_k}, \ \ g_k=g_{\sigma_k}, \ \ d_k=d_{\sigma_k}, \ \mbox{ and } \ X_k=(X, d_k).
\]
 
\begin{proposition}\label{LinearApprox}
Let $u:(B_1(0),g) \to (X,d)$ be a harmonic map where $X$ is of type $W$ and the dimension of $X$ is at least $2$. For  $p \in M$ with $\text{Ord}^u(p)=1$, let $u_\sigma$ be the blow up maps at $p$ defined by \eqref{eq:usigmamap}. Then there exists 
\begin{itemize}
\item a sequence $\sigma_k \to 0$,
\item a sequence of  homogeneous degree 1 harmonic maps $L_k:\R^n \to A_k \subset X_k$, where $A_k$ is an apartment in $X_k$, and 
\item $r_0 \in (0,1)$
\end{itemize}
 such that the following properties are satisfied:
\begin{itemize}
\item[(i)]  
$(X_k, A_k, L_k)$ is an $(X_\omega, A_\omega, L_\omega)$-triple where 
$
(X_\omega,d_\omega)=\wlim \, (X,d_k, \star_k)$ and  $L_\omega=\wlim L_k:\R^n \rightarrow A_\omega \subset X_\omega$.  Here,  $\star_k=u_k(0)=u(p)$ and $A_\omega$ is an apartment of the building $X_\omega$. 
 \item[(ii)]  
 The energy density measures  and the directional energy density measures  of $u_k$ converge to that of $u_\omega$ in $B_{r_0}(0)$.
\item[(iii)] $\displaystyle{\lim_{k \rightarrow \infty} \sup_{x \in \B_{r_0}(0)}d_k( u_k(x), L_k(x)) = 0}$.
\end{itemize}
\end{proposition}

\begin{proof}
As explained in Section~\ref{subsec:rescaling}, a subsequence of the $u_k$ (which we still denote by $u_k$) converges locally uniformly to a tangent map $u_*:B_1(0) \to (X_*,d_*)$. Let  $(X_\omega,d_\omega):=\wlim \, (X,d_k, \star_k)$ and  $u_\omega:=\wlim u_k$. By  Section~\ref{sec:limitspace}, we can assume $u_*=u_\omega$. Thus, the (directional) energy density measures   of $u_k$ converge to those of $u_\omega$, and  $u_\omega$ is a non-constant homogeneous degree 1 harmonic map.

By \cite[Theorem 5.1.1]{kleiner-leeb}, $X_\omega$ is a Euclidean building of type $W$.
By Proposition~\ref{prop:flats}, there exists $r_0 \in (0,1)$ and an apartment $A_\omega \subset X_\omega$ such that $u_\omega|B_{r_0}(0)$ can be extended  as a homogeneous degree 1 harmonic map $L_\omega:\R^n \rightarrow A_\omega \subset X_\omega$. 
Consequently, we obtain assertion (ii) about  the convergence of energies in $B_{r_0}(0)$ where $u_\omega=L_\omega$.

A chart  $\iota_\omega=[(\iota_1, \iota_2, \dots)]:\R^N \to A_\omega$  is an ultralimit of the sequence of charts $\iota_k: \R^N \rightarrow A_k$  where each $A_k$ is an apartment of the building $X_k$ (cf~\cite[Proof of Theorem 5.1.1]{kleiner-leeb}).
Let
\begin{equation}\label{eq:Lk}
L_k:= \iota_k \circ \iota_\omega^{-1} \circ L_\omega: \R^n \to A_k \subset X_k.
\end{equation}
For $x \in \R^n$, choose  $y \in \R^N$ satisfying $L_\omega(x) = \iota_\omega(y)$. Then
\begin{align*}
[(L_k(x))] =[(\iota_k \circ \iota_\omega^{-1} \circ L_\omega(x))]
=[(\iota_k(y))]
=\iota_\omega(y)
&=L_\omega(x).
\end{align*} 
Thus, $\wlim L_k= L_\omega$. 
The homogeneity of $L_k$ follows immediately from the definition since $\iota_k$ and $\iota_\omega$ are isometric embeddings.  To see that  $(X_k, A_k, L_k)$ is a $(X_\omega, A_\omega, L_\omega)$-triple,  note that the isometry $\phi:A_\omega \to A_k$ is given simply by $\phi= \iota_k \circ \iota_\omega^{-1}$.   This completes the proof of (i).
Finally, to prove (iii), we take a further subsequence according to Lemma~\ref{lem:ukLk} below.
\end{proof}

\begin{lemma}\label{lem:ukLk}
There exists a subsequence $(k_j)$ such that
\[
\lim_{j \rightarrow \infty} \sup_{\B_{r_0}(0)}d_{k_j}(u_{k_j}(x), L_{k_j}(x)) = 0.
\]
\end{lemma}

\begin{proof}
By the definition of $L_k$ (cf~(\ref{eq:Lk})), the uniform energy bound for $u_k$, and \cite[Theorem 2.4.6]{korevaar-schoen1}, there exists an $M>0$ independent of $k \in \N$ such that for all $x,y \in \B_{r_0}(0)$,
\[
d_k(L_k(x), L_k(y))\leq M|x-y| \quad \quad \text{and} \quad \quad d_k(u_k(x), u_k(y))\leq M|x-y|.
\]  

For each $j \in \mathbb N$,  let $\Omega_j:=\{x_1, x_2, \dots, x_{N_j}\}\subset B_{r_0}(0)$ be a finite set such that, for all $x \in \B_{r_0}(0)$, there exists $x_\alpha \in \Omega_j$ with  $|x-x_\alpha|< \frac 1{3Mj}$. Define
\[
S_\alpha:=\{ k \in \N: d_k(u_k(x_\alpha), L_k(x_\alpha))< \frac 1{3j}\} \subset \N.
\]As 
\[
0=d_\omega(u_\omega(x_\alpha), u_\omega(x_\alpha))= d_\omega(u_\omega(x_\alpha), L_\omega(x_\alpha))= \wlim d_k(u_k(x_\alpha), L_k(x_\alpha)),
\]we see that $\omega(S_\alpha)=1$ for each $\alpha \in \{1, \dots, N_j\}$. (For more information on the ultrafilter $\omega$, consult \cite[Section 2.4.1]{kleiner-leeb}.) Therefore 
\[
\omega\left(\bigcap_{\alpha=1}^{N_j} S_\alpha\right)=1.
\]Choose $k_j \in \bigcap_{\alpha=1}^{N_j} S_\alpha$ inductively such that $k_1 \geq 1$ and $k_{j+1} > k_j$.

 For $x \in B_{r_0}(0)$, choose $x_\alpha \in \Omega_j$ such that $|x-x_\alpha|< \frac 1{3Mj}$. Then
\begin{align*}
d_{k_j}(u_{k_j}(x), L_{k_j}(x)) &\leq d_{k_j}(u_{k_j}(x), u_{k_j}(x_\alpha))+d_{k_j}(u_{k_j}(x_\alpha), L_{k_j}(x_\alpha))+d_{k_j}(L_{k_j}(x), L_{k_j}(x_\alpha))\\
& < M|x-x_\alpha| + \frac 1{3j} + M|x-x_\alpha|
\leq \frac 1j.
\end{align*}
This verifies that $\sup_{\B_{r_0}(0)}d_{k_j}(u_{k_j}(x), L_{k_j}(x))\rightarrow  0$ as $j \rightarrow \infty$. 
\end{proof}


\section{Local Product Structure}\label{sec:noregFin}

We are now ready to prove the local product structure  of a harmonic map at an order 1 point, as stated in Theorem \ref{GS Theorem 5.1}.  This  is proven in \cite[Theorem 5.1]{gromov-schoen} for locally finite Euclidean buildings, and the proof here closely follows  their proof. There is a key difference --  the use of Proposition~\ref{Sun Lemma 12 Collapsed}.  

\begin{definition}
For a map $u:B_1(0)\to X$, given $B_\sigma(x) \subset B_1(0)$, the {\bf remainder} $R^u(x,\sigma)$ is
\[
R^u(x,\sigma)=\inf_L\sup_{B_\sigma(x)}d(u(y),L(y))
\]
where the infimum is taken over homogeneous degree 1 maps  about $x$.
\end{definition}

\begin{definition}\label{def:aomegaw}
Fix a Riemannian domain $\Omega$ and a finite reflection group $W$. Consider the collection of maps $$A_{\Omega,W}:=\{u:\Omega\to X\,|\,u\text{ is a harmonic map, }X\text{ is a Euclidean building of type }W\}.$$   Let $K\subset\Omega$ be a compact subset. The class $A_{\Omega,W}$ is a {\bf $K$-intrinsically differentiable class} if there are constants $\sigma_0,c>0$ and $\beta\in(0,1]$ (depending only on $K,\Omega, W$) so that, for any $u\in A_{\Omega,W}$, any $x\in K$, and for all $0<\sigma<\sigma_0$ such that $B_\sigma(x) \subset \Omega$, 
\[
R^u(x,\sigma)\leq c\sigma^{1+\beta}R^u(x,\sigma_0).
\]
\end{definition}

\begin{definition} \label{def:essreg}
A Euclidean building $X$ of type $W$ is {\bf essentially regular} if, for every Riemannian domain $\Omega$, and every compact $K\subset\Omega$, $A_{\Omega,W}$ is a $K$-intrinsically differentiable class.
\end{definition}

\begin{theorem}\label{GS Theorem 5.1}
Assume $\Waff=\rho^{-1}(W)$.
Fix $E_0>0$,  $r_0 \in (0,1)$ and a  homogeneous degree 1 map 
\[
L:\R^n \rightarrow A \subset X
\]
where the dimension of $X$ is at least $2$ and $A$ is an apartment of $X$.   
Then there exists $\delta_0=\delta_0(E_0, r_0, X, A, L)>0$ with the following property: \\
\\
Let  $(X',A',L')$  be an $(X,A,L)$-triple and let $P_{F'}=F' \times Y$ be the union of all flats parallel to $F':=L'(\R^n)\simeq \R^m$.
Let $u:(B_1(0),g)\to X'$ be a finite energy harmonic map with ${^gE}^u[B_1(0)]<E_0$. 
If 
\begin{itemize}
\item $P_{F'}$ is  essentially regular, 
\item $\|g-\deltae\|_{C^2(B_1(0))}<\delta_0$, and
\item $\sup_{B_{r_0}(0)}d(u,L')<\delta_0$,
\end{itemize}
then $u(B_{r_0/4}(0))\subset P_{F'}$.
\end{theorem}

\begin{remark}
The condition that $\Waff=\rho^{-1}(W)$ is a technical one which we introduce here for convenience. Given any building $X$ of type $W$, we can simply enlarge the group $\Waff$ to satisfy this hypothesis. (Recall Lemma \ref{biggeraffinegp}.)

The condition on $\Waff$ is not natural when thinking about a single building, since in principle $X$ may no longer be thick. But in the proof below, we need to replace $L'$ by a homogeneous degree 1 map $L_0$ which agrees with $u$ at a particular point and we need to find an apartment $A_0$ such that $(X', A_0, L_0)$ is an $(X,A,L)$-triple to invoke Proposition \ref{Sun Lemma 12 Collapsed}. This is straightforward to accomplish when $\Waff = \rho^{-1}(W)$.
\end{remark}

\begin{proof}
We will choose an $\epsilon_1>0$ toward the conclusion of the proof, but for the moment demand only that $100\epsilon_1<\mu_0(B_{\frac {r_0}4}(0))$. Now choose $\eta=\eta(\epsilon_1,E_0, r_0, X, A, L)>0$ as in Proposition~\ref{Sun Lemma 12 Collapsed}. Choose $\delta_0>0$ such that $2(1+4 \theta^{-1} )\delta_0 \leq\eta$, where $\theta \in (0, 1/4]$ will be chosen later. 
This bound implies, in particular, that $$\mu_0\{x\in B_{\frac{r_0}{2}}(0):u(x)\notin P_{F'}\}< \epsilon_1<\mu_0(B_{\frac{r_0}{4}}(0)).$$  Hence, there {\em are} points of $B_{\frac {r_0}2}(0)$ with $u(x)\in P_{F'}$.

Let $x_0\in B_{\frac{r_0}{2}}(0)$ be one such point.  
By the third bullet point, $d(u(x_0),L'(x_0))<\delta_0$. Hence, there exists a flat $F_0$ parallel to $F'$ such that $d(F', F_0) < \delta_0$ and $u(x_0) \in F_0$. Using the product structure $F' \times Y\simeq P_{F'} $ of Lemma~\ref{Xnot}, let $L'(x_0) = (0,y')\in F' \times \{y'\} \simeq F'$ and let $t \in F', y_0 \in Y$ such that $u(x_0)= (t,y_0) \in F' \times \{y_0\} \simeq F_0$.  Let $A_Y$ be an apartment of $Y$ containing $y'$ and $y_0$. Thus, $F'$ and $F_0$ are both contained in the apartment $F' \times A_Y\simeq A_0$  of $P_{F'}$. Let  $\tau:A_0 \to A_0$ be a translation which takes $(0,y')$ to $(t,y_0)$ and define $L_0=\tau \circ L'$.  Then $L_0(x_0)=u(x_0)$ and $\sup_{B_{\frac{r_0}2} (x_0)}d(L_0,u)< 2\delta_0$.

 Following the ideas in Remark \ref{rem:thetanot}, let $\iota_{A'}:\R^N \to A'$ be a chart in $\varphi \cdot \mathcal A$ where $\varphi:\R^N \to \R^N$ is an orthogonal transformation such that $\iota_{A'}(\R^m \times \{(0, \dots, 0)\}) = F'$. 
 Let $\iota_{A_0}$ denote a chart of $\varphi \cdot \mathcal A$ such that the restriction of $\iota_{A_0} \circ\iota_{A'}^{-1}$  to $\R^m$ is the identity map. Let $\hat \phi:A \to A_0$ where $\hat \phi:=\tau \circ \iota_{A_0} \circ\iota_{A'}^{-1}\circ \phi$. (Here $\phi:A \to A'$ is as in Definition \ref{def:triple}.) Since $\Waff = \rho^{-1}(W)$, $\hat \phi$ satisfies the necessary conditions which make $(X', A_0, L_0)$ an $(X,A,L)$-triple. 

Fix normal coordinates centered at $x_0$. For a map $f:B_{\frac{r_0}{2}}(x_0)\to (X',d')$, define ${^if}:(B_{\frac {r_0}{2}}(0), g_i)\to(X',d_i)$ where ${^if}(x):=f(\theta^ix)$, $g_i(x) = \theta^{-i}g(\theta^ix)$, and $d_i:=\theta^{-i}d$. The uniform Lipschitz bounds on $u$ on $B_{\frac{r_0}{2}}(x_0)$, which depend only on ${^gE}^u[B_1(0)]$, imply uniform energy bounds on ${^{g_i}E}^{{^iu}}[B_{\frac{r_0}{2}}(0)]$.

Our inductive claim is as follows:
\noindent For $\theta>0$ sufficiently small and for each $i \in \mathbb Z_{\geq 0}$, 
\begin{enumerate}
\item there exists $\delta_i>0$ such that $$\sup_{B_{\frac{r_0}4}(0)}d_i({^iu},{^iL}_0)\leq 2 \delta_i\leq 2\Big(1+2\theta^{-1}\sum_{j=0}^{i-1}2^{-j}\Big)\delta_0\leq 2(1+4\theta^{-1})\delta_0,$$where for $i=0$ we presume that $\sum_{j=0}^{-1}2^{-j}=0$, 
\item there exists a homogeneous degree $1$ map \[M_i:(B_{\frac{r_0}4}(0),g_i)\to (P_{F'}, d_i)\] so that \[\sup_{B_{\frac{r_0}4}(0)}d_i({^iu},M_i)=D_i\leq2^{-i+1}\delta_0.\]
\end{enumerate}
For the base case, $i=0$, we set $M_0=L_0$ and observe that with $2\delta_0$ in the first inequality, $D_0=2\delta_0$ in the second, the claimed bounds hold.

We now suppose that these inequalities hold for some $i$ and argue that they hold for $i+1$. First, we observe $u(x_0)={^iu}(0)={^iL}_0(0)$, so that
\[
d_i({^iL}_0(0),M_i(0))\leq D_i.
\]
While the triangle inequality implies
\[
\max_{\partial B_{\frac{r_0}4}(0)} d_i({^iL}_0,M_i)\leq 2\delta_i+D_i,
\]
together with the homogeneity of ${^iL}_0,M_i$ we have
\[
\sup_{B_{\frac{\theta r_0}4}(0)}d_i({^iL}_0,M_i)\leq 2\theta \delta_i+D_i.
\]
Combining this with the assumed bounds on $d_i({^i}u,M_i)$, we see that
\[
\sup_{B_{\frac{\theta r_0}4}(0)} d_i({^iu},{^iL}_0)\leq 2\theta\delta_i+2D_i.
\]
In particular, considering the rescaled distance $d_{i+1}$ on $B_{\frac{r_0}2}(0)$, we have that
\[
\sup_{B_{\frac{r_0}4}(0)}d_{i+1}({^{i+1}u},{^{i+1}L}_0)\leq 2\delta_i+2\theta^{-1}D_i=:2 \delta_{i+1}.
\]
The assumed bounds on $\delta_i,D_i$ immediately imply the needed bound on $\delta_{i+1}$.

We now turn our attention to point (2). 
Since $(X',A_0, L_0)$ is an $(X,A,L)$-triple
and $2\delta_i \leq 2(1+4\theta^{-1})\delta_0\leq\eta$, applying Proposition~\ref{Sun Lemma 12 Collapsed}, 
\[
\mu_0^n\{x\in B_{\frac {r_0}{4}}(0):{^iu}(x)\notin P_{F'}\}<\epsilon_1.
\]
In particular, there is at least one radius $r\in[r_0/8, r_0/4]$ so that
\[
\mu_0^{n-1}\{x\in\partial B_r(0):{^iu}(x)\notin P_{F'}\}<8\epsilon_1/r_0.
\]
Let $\pi:X' \to P_{F'}$ denote the closest point projection, and let $v$ be the energy-minimizing map $v:B_r(0)\to P_{F'}$ with $v=\pi\circ{^iu}$ on $\partial B_r(0)$.  We shall first show that $v$ is very close to ${^iu}$.  Point (2) of the inductive hypothesis implies that ${^iu}$ is at distance at most $D_i$ from $P_{F'}$ on $B_{\frac{r_0}4}(0)$, and the measure estimate implies that in fact
\[
\int_{\partial B_r(0)}d_i({^iu},v)d\Sigma\leq8\epsilon_1D_i/r_0.
\]
Since \cite[Equation (2.2)]{gromov-schoen} can be extended to all Euclidean buildings, we may follow the proof of \cite[Lemma 5.3]{gromov-schoen} to see that $d_i({^iu},v)$ is subharmonic. This implies that for some constant $c_1$ depending only on the domain,
\begin{equation}\label{eq:subharm}
\sup_{B_{\frac{r_0}{16}}(0)}d_i({^iu},v)\leq c_1\epsilon_1D_i/r_0.
\end{equation}

Now, by hypothesis $P_{F'}$ is essentially regular and thus, $v$ is $B_r(0)$-intrinsically differentiable. It follows that there exists a homogeneous degree 1 map $ \widetilde M$ with $\widetilde M(0)=v(0)$, such that for any homogeneous degree $1$ map $M$ and sufficiently small $\theta>0$, 
\[
\sup_{B_{\frac{\theta r_0}4}(0)}d_i(v,\widetilde M)\leq c_2\theta^{1+\beta}\sup_{B_{\frac{r_0}{16}}(0)}d_i(v,M).
\]
Here $c_2$ depends on the domains $K=\overline{B_{\frac{r_0}{16}}(0)}$, $\Omega=B_{\frac{r_0}8}(0)$, and on the total energy of $v$ (and hence of $u$).  We remark at this point that, although $c_2$  also depends on the metric on $B_{\frac{r_0}8}(0)$, since these metrics are being rescaled towards the Euclidean metric, the $c_2$ involved will improve as the induction continues. 

Applying this for $M=M_i$, and using the triangle inequality, the inductive assumption, and (\ref{eq:subharm}),
\begin{align*}
\sup_{B_{\frac{\theta r_0}4}(0)}d_i(v,\widetilde M)&\leq c_2\theta^{1+\beta}\sup_{B_{\frac{r_0}{16}}(0)}d_i(v,M_i)\\
&\leq c_2\theta^{1+\beta}\sup_{B_{\frac{r_0}{16}}(0)}\left(d_i(M_i,{^iu})+d_i({^iu},v)\right)\\
&\leq c_2\theta^{1+\beta}D_i(1+c_1\epsilon_1/r_0).
\end{align*}
Again applying the triangle inequality and (\ref{eq:subharm}),
\[
\sup_{B_{\frac{\theta r_0}4}(0)}d_i({^iu},\widetilde M)\leq (c_1\epsilon_1/r_0+c_2\theta^{1+\beta}+c_1c_2\theta^{1+\beta}\epsilon_1/r_0)D_i.
\]
Now, as long as we take $\theta,\epsilon_1$ sufficiently small, we may absorb the final term into the former two, so
\begin{equation}\label{eq:keyind}
\sup_{B_{\frac{\theta r_0}4}(0)}d_i(^iu,{{\widetilde M}})\leq2(c_1\epsilon_1/r_0+c_2\theta^{1+\beta})D_i.
\end{equation}
This immediately tells us that
\[
\sup_{B_{\frac {r_0}4}(0)}d_{i+1}({^{i+1}u},{^1{\widetilde M}})\leq2\theta^{-1}(c_1\epsilon_1/r_0+c_2\theta^{1+\beta})D_i.
\]
Setting $M_{i+1}:={^1{\widetilde M}}$ completes the inductive step, so long as
\[
2\theta^{-1}(c_1\epsilon_1/r_0+c_2\theta^{1+\beta})\leq\frac{1}{2}.
\]
To choose $\theta, \epsilon_1$ appropriately, first observe that the constants $c_1,c_2$ depend on the domain (ie its dimension and metric---in particular, how far it is from Euclidean) and on the total energy of the map $^iu$.  The dimension is constant, the metric converges to Euclidean as $i\to\infty$, and the bound on the energy of $u$ implies that we have uniform energy bounds on the ${^iu}$.  Hence, these constants do not depend on $i$, or (crucially) upon $\theta,\epsilon_1$.

We first choose $\theta\in(0, 1/4]$ so that $2c_2\theta^\beta\leq\frac{1}{4}$, then decrease $\epsilon_1$ if necessary so that $2c_1\theta^{-1}\epsilon_1\leq\frac{r_0}4$. Since, at the outset, we chose $\delta_0$ such that $2(1+4\theta^{-1})\delta_0< \eta$ and at each step of the induction argument, $\sup_{B_{\frac{r_0}4}(0)}d({^iu},{^iL}_0)\leq 2(1+4\theta^{-1})\delta_0$, this bound on $\delta_0$ 
 ensures that the inductive argument works at all stages.  In particular, at each stage, we have $\mu_0\{x\in B_{\frac {r_0}{4}}(0):{^i}u(x)\notin P_{F'}\}<\epsilon_1$.

Finally, we show that $u(B_{\frac{r_0}{4}}(0))\subset P_{F'}$. Suppose, for the sake of contradiction, that there is some point $y\in B_{\frac{r_0}{4}}(0)$ so that $u(y)\notin P_{F'}$.  Recall that the open set $U=\{x\in B_{\frac{r_0}{4}}(0):u(x)\notin P_{F'}\}$ has $\mu_0(U)<\epsilon_1<\mu_0(B_{\frac{r_0}{4}}(0))$.  Hence we may in fact choose $y \in B_{\frac{r_0}{4}}(0)$ so that for some $0<r<\frac{r_0}{4}$, $B_r(y)\subseteq U$ and there exists $x_0 \in u^{-1}(P_{F'}) \cap \partial B_r(y)$. We observe that, in the limit as $\sigma\to0$, at least half of the ball $B_\sigma(x_0)$ lies in $U$.

But on the other hand, by the above inductive argument, at the scale $\theta^i$, we have that $d_i(^iu,{^iL}_0)$ is small enough that
\[
\mu_0\{x\in B_{\frac{r_0}{4}}(0):{^iu}(x)\notin P_{F'}\}<\epsilon_1
\]
and in particular for $i$ large enough no more than $\frac{1}{4}$ of the ball $B_{\frac{r_0}{4\theta^i}}(x_0)$ lies in $U$.  This contradiction allows us to conclude that there are no points of $U$ in $B_{\frac{r_0}4}(0)$. That is, $u(B_{\frac{r_0}4}(0))\subset P_{F'}$.
\end{proof}

We are now in a position to state two key ingredients in the proofs of the main theorems, namely items (1) and (2) below which provide a gap theorem for harmonic maps into all Euclidean buildings and demonstrate that such maps have a local product structure at order 1 points. The theorem mirrors that of \cite[Theorem 6.3]{gromov-schoen} and the proof is quite similar, though 
our proof applies Theorem~\ref{GS Theorem 5.1} in place of \cite[Theorem 5.1]{gromov-schoen}.
  Because of the similarities, we defer the proof to Appendix \ref{sec:essreg}.

\begin{theorem}\label{thm:ordgap}
Let $X$ be a Euclidean building of type $W$.  Then we have the following:
\begin{enumerate}
\item There is a constant $\epsilon$ depending on $n,g$ and $W$ 
such that for a harmonic map $u:(\Omega^n,g)\to X$ and any $p\in\Omega$, either $\text{Ord}^u(p)=1$ or $\text{Ord}^u(p)\geq1+\epsilon$.
\item When $\text{Ord}^u(p)=1$, there exists a subbuilding $\PF\simeq \FY$ and an $r >0$ such that $F\simeq\R^m$ for some $m \in \{1, \dots, \min\{n,N\}\}$, $Y=Y^{N-m}$ is a building of lower dimension, and $u|B_r(p)$ decomposes into two harmonic maps where $u|{B_r(p)}=(u_1,u_2):B_r(p)\to \PF$, such that $u_1:B_r(p) \to \R^m$ is a harmonic map of rank $m$ and $u_2:B_r(p) \to Y$ satisfies either $\text{Ord}^{u_2}(p)\geq 1+\epsilon$ or $u_2$ is a constant map.
\item $X$ is essentially regular. 
\end{enumerate}
\end{theorem}

\section{Proof of the main theorems}
\label{sec:proofsofrigdity}

Because of the analysis provided in the previous sections, the proofs of the main theorems of this paper (stated in the introduction) follow from adapting the proofs in \cite{gromov-schoen} for locally finite buildings. In this section, we provide the necessary adjustment to their arguments.

\subsection{Proof of Theorem~\ref{maintechnical}}
\begin{definition}\label{def:singular}
Let $u: \Omega \rightarrow X$  be a harmonic map from a Riemannian domain into a Euclidean building of type $W$.  A point $p \in \Omega$ is called a {\bf regular point} if there exists a neighborhood $U$ of $p$ and an apartment $A$ of $X$ such that $u(U) \subset A$.  Otherwise $p$ is called a {\bf singular point}.  The {\bf singular set} $\mathcal S(u)$  is  the set of all singular points.  
\end{definition}

Define
\begin{eqnarray*}
{\mathcal S}_{>1}(u) & = & \{p \in \mathcal S(u):  Ord^u(p) >1\}
\\
{\mathcal S}_{=1}(u) & = & \{p \in \mathcal S(u):  Ord^u(p)=1\}
\\
{\mathcal S}_0(u)  & = &  \{p \in \Omega:  Ord^u(p) >1\}
\end{eqnarray*}
In particular,
\[
\mathcal S(u)=\mathcal S_{=1}(u) \cup \mathcal S_{>1}(u) \ \ \mbox{ and } \ \ \mathcal S_{>1}(u) \subset \mathcal S_0(u).
\]
\begin{lemma} \label{higherorderpoints}
If $n=\dim \Omega$, then {${\mathcal S}_0(u)$ is a closed set with}
$
\dim_{\mathcal H}({\mathcal S}_0(u)) \leq n-2.
$
\end{lemma}

\begin{proof}
This follows from a slight modification of the proof in \cite[paragraph after Lemma 6.5]{gromov-schoen}. The main modification to the argument is in fact in the invoking of the conclusion of \cite[Lemma 6.5]{gromov-schoen}. The conclusion of \cite[Lemma 6.5]{gromov-schoen} holds in this setting for the (Gromov-Schoen) blow up maps $u_k$ and $u_\omega~:~=~\wlim u_k$ since the local uniform convergence in the pullback sense given by Remark \ref{KLTheorem5.1} and properties of the order function imply that for $x_k \to x$, where $x_k \in \mathcal S_0(u_k)$, $\limsup_{k \to \infty} \text{Ord}^{u_k}(x_k) \leq \text{Ord}^{u_\omega}(x)$ and thus by Theorem \ref{thm:ordgap} item (1), $x \in \mathcal S_0(u_\omega)$. 
\end{proof}

\begin{proof1}
This is an easy consequence of Theorem~\ref{thm:ordgap}.  For the sake of completeness, we will include the proof which  involves  an inductive argument on the dimension of $X$.  The one dimensional case was proved in  \cite{sun}.  Now assume that the assertion is true for dimensions less than $N$.  

Suppose $p \in \mathcal S_{=1}(u)$. Then Theorem~\ref{thm:ordgap} item (2) asserts that there exists an $r>0$ and a lower dimensional subbuilding $(Y,d)$ such that for $\sigma \in (0,r]$, $u_\sigma=(u_{\sigma, 1}, u_{\sigma, 2})$ where  $u_{\sigma, 1}:B_1(0) \to \R^m$ and $u_{\sigma, 2}: B_1(0) \to (Y,d_\sigma)$.  Thus, by the inductive hypothesis, there exists a $\sigma_p>0$ such that
$$
\dim_{\mathcal H}(\mathcal S(u) \cap B_{\sigma_p}(p))=\dim_{\mathcal H}(\mathcal S(u_{\sigma_p})) = \dim_{\mathcal H}(\mathcal S(u_{\sigma_p, 2})) \leq n-2.
$$
 Now cover $\mathcal S_{=1}(u)$ by balls $\{B_{\sigma_p}(p)\}_{p \in \mathcal S_{=1}(u)}$ and refine this cover to a countable subcover $\{B_{\sigma_{p_j}}(p_j)\}$. Since for all $t>n-2$, $\mathcal H^t(B_{\sigma_{p_j}}(p_j))= 0$ and $\mathcal H^t$ is countably additive,  $\mathcal H^t(\mathcal S_{=1}(u))=0$  for all $t>n-2$. Thus, $\dim_{\mathcal H}(\mathcal S_{=1}(u)) \leq n-2$.  Combined with Lemma~\ref{higherorderpoints}, we conclude {that $\mathcal S(u)$ is a closed set with} $\dim_{\mathcal H}(\mathcal S(u)) \leq n-2$.
\end{proof1}

The following corollary of Theorem~\ref{maintechnical} will be important in the proof of the rigidity theorems.

\begin{corollary} \label{thm:maintech'}
For any compact subdomain $\Omega_1$ of $\Omega$, there is a sequence of Lipschitz functions $\{\psi_i\}$ such that  $\psi_i\equiv 0$ in a neighborhood of $\mathcal S(u) \cap \overline{\Omega_1}$, $0 \leq \psi_i \leq 1$, $\psi_i(x) \rightarrow 1$ for all $x \in \Omega_1 \backslash \mathcal S_1$, and  
\[
\lim_{i \rightarrow \infty} \int_\Omega |\nabla du| |\nabla \psi_i| \, d\mu =0.
\]
\end{corollary}

\begin{proof}
The proof follows from \cite[p.~227, third paragraph]{gromov-schoen}, adapted to non-locally finite case using Theorem~\ref{thm:ordgap} and Theorem~\ref{maintechnical}.
\end{proof}

\subsection{Proofs of Theorem~\ref{maintheorem} and Theorem~\ref{maintheorem'}}
First we recall the definition of Finite Rank (FR) spaces given in  \cite{korevaar-schoen3}.
\begin{definition}
A CAT(0) space $X$ is an {\bf FR-space} if there exists $\epsilon_0>0$ and $R_0>1$ such that any subset of $X$ with diameter $D>R_0$ is contained in a ball of radius $(1-\epsilon_0)D/\sqrt 2$.
\end{definition}
Since every Euclidean building has finite geometric dimension, \cite[Theorem 1.3]{caprace-lytchak} implies that every Euclidean building is an FR-space. 
Korevaar and Schoen prove an existence theorem for harmonic maps to FR-spaces.
\begin{theorem}[\cite{korevaar-schoen3} Theorem 1]
\label{thm:existenceKS3}
Let $\pi_1(M)$ be a fundamental group of a compact Riemannian manifold $M$ and $\rho$ be an isometric action of $\pi_1(M)$ on an FR-space $X$. Either there exists an invariant equivalence class of rays or there exists a $\rho$-equivariant harmonic map $u:\widetilde M \rightarrow X$.
\end{theorem}

\begin{proposition} \label{existence}
Let $\widetilde M$, $X$ and $\rho$ be as in Theorem~\ref{maintheorem}. Then there exists a $\rho$-equivariant harmonic map $u:\widetilde M \rightarrow X$.
\end{proposition}

\begin{proof}
If the rank of $\widetilde M$ is $\geq 2$, then the assertion follows from Theorem~\ref{thm:existenceKS3} and the assumption that $\rho$ does not fix a point at infinity (see also \cite[Corollary 3.6]{jost}).  Moreover, it is straightforward to verify that  we can replace the assumption that $M$ is compact  in Theorem~\ref{thm:existenceKS3} by the assumption that $M$ is of finite volume and there exists a $\rho$-equivariant finite energy locally Lipschitz map $\widetilde M \rightarrow X$.  
The existence of a finite energy locally Lipschitz map follows from \cite[Lemma 8.1]{gromov-schoen}, noting that the proof does not need the target building to be locally finite.  Thus, the assertion for the case when the rank of $\widetilde M$ is 1 also follows the same way.
\end{proof}

\begin{proof_of_main_theorem2}
The proof follows from applying the Bochner method to the map $u$ given by Proposition~\ref{existence}.  More specifically, for lattices in rank 1 groups we can  follow the proofs of  \cite[Theorems 7.2 and~7.4]{gromov-schoen}, using Theorem~\ref{maintechnical} and Corollary~\ref{thm:maintech'} in the appropriate places, to show that $u$ is constant.

For higher rank groups we follow the proof of \cite[Lemma 15]{daskal-meseGAFA} to show that $u$ is totally geodesic in a neighborhood of every regular point. Then, using Theorem \ref{maintechnical} and the arguments in the first part of the proof of \cite[Theorem 1]{daskal-meseGAFA}, we conclude that $u$ is totally geodesic. (Notice that part of their proof does not require any additional hypotheses on $\rho$.) In particular, this implies that $\text{Ord}^u(p)=1$.  By Theorem \ref{thm:ordgap},  there exists a flat $F$ and an $r>0$ such that $u|B_r(p)=(u_1,u_2):B_r(p) \to F \times Y$ where both $u_1$ and $u_2$ are totally geodesic and $\text{Ord}^{u_2}(p)>1$.  Thus, $u_2$ is a constant map which in turn implies $u$ is a smooth, totally geodesic harmonic map into $F  \subset A \simeq \R^N$ for some apartment $A \subset X$. It thus follows that $u$ must be a constant map.
\end{proof_of_main_theorem2}

\begin{proof_of_main_theorem3}
Using Proposition \ref{existence}, Theorem~\ref{maintechnical} and Corollary~\ref{thm:maintech'}, we can follow the proofs of  \cite[Theorems 7.2 and 7.3]{gromov-schoen}.
\end{proof_of_main_theorem3}

\appendix
\section*{Appendix}

\renewcommand{\thesection}{A} 
\section{Proof of Theorem \ref{thm:ordgap}}\label{app:proofof6}
In this section we prove Theorem \ref{thm:ordgap}. We begin with a few technical lemmas, then prove essential regularity for trees and finally, in Section \ref{sec:essreg} we prove the theorem.

The following lemma is useful for showing that complexes are essentially regular; the computations are standard but we include them for completeness.

\begin{lemma}\label{lem:intdifftheta}
Suppose that there is a $\tau_0>0$ and a $\theta\in(0,1/2]$ so that for any $u\in A_{\Omega,W}$, any $x_0\in K$, and any $\sigma\in(0,\tau_0]$, 
\[
R^u(x_0,\theta\sigma)\leq\frac\theta2R^u(x_0,\sigma).
\]
Then $A_{\Omega,W}$ is a $K$-intrinsically differentiable class.
\end{lemma}
\begin{proof}
We will use $\tau_0$ to be the $\sigma_0$ in the definition of intrinsic differentiability.  We observe that by a quick induction argument (and the monotonicity of $R^u(x,\sigma)$ in its second argument), if $\sigma\in(\theta^n\tau_0,\theta^{n-1}\tau_0]$, then
\[
R^u(x_0,\sigma)\leq\frac{\theta^{n-1}}{2^{n-1}}R^u(x_0,\tau_0).
\]
Choose $\beta\in (0,1]$ such that $\frac12=\theta^\beta$ and set $c=\frac{2}{\theta\tau_0^{1+\beta}}$. If $\sigma\in(\theta^n\tau_0,\theta^{n-1}\tau_0]$, then 
\[
R^u(x_0,\sigma)\leq c\tau_0^{1+\beta}\frac{\theta^n}{2^n}R^u(x_0,\tau_0)=c\tau_0^{1+\beta}(\theta^n)^{1+\beta}R^u(x_0,\tau_0).
\]
Then, because $\theta^n\tau_0<\sigma$,
\[
R^u(x_0,\sigma)\leq c(\tau_0\theta^n)^{1+\beta}R^u(x_0,\tau_0)\leq c\sigma^{1+\beta}R^u(x_0,\tau_0).
\]
\end{proof}

The next lemma generalizes the  classical Reverse Poincar\'e inequality. 

\begin{lemma} \label{lem:ReversePoincare}
There exists a constant $C>0$ such that, for a harmonic map $u:(B_1(0),g) \rightarrow X$ and $Q \in X$,
\[
\int_{B_{1/2}(0)} |\nabla u|^2 \, d\mu_g\leq C \int_{B_1(0)} d^2(u,Q)\, d\mu_g.
\]
\end{lemma}

\begin{proof}
By \cite[Proposition 2.2]{gromov-schoen},
\[
2 \int_{B_1(0)} |\nabla u|^2 \varphi \, d\mu_g \leq - \int_{B_1(0)} \nabla d^2 (u,Q) \cdot \nabla \varphi  \, d\mu_g,\ \ \forall \varphi\in C^\infty_c(B_1(0)), \varphi \geq 0.
\]
Let  $\varphi=\eta^2$ be a radial test function with $\eta\equiv 1$ in $B_{1/2}(0)$, $\eta \equiv 0$ in $B_1(0) \backslash B_{3/4}(0)$, and $|\nabla \eta| \leq 8$.
Then
\begin{eqnarray*}
 \int_{B_1(0)} |\nabla u|^2 \eta^2  \, d\mu_g
& \leq &
2 \int_{B_1(0)} | d(u,Q) \nabla \eta |\,| \eta \nabla d(u,Q) |\, d\mu_g
\\
& \leq & 2\left( \int_{B_1(0)} d^2(u,Q) |\nabla \eta|^2  \, d\mu_g\right)^{1/2} \left( \int_{B_{1}(0)} |\nabla d(u,Q) |^2 \eta^2 \, d\mu_g\right)^{1/2}
\\
& \leq & 16 \left( \int_{B_1(0)} d^2(u,Q)\, d\mu_g\right)^{1/2} \left( \int_{B_{1}(0)} |\nabla u |^2\, \eta^2 d\mu_g\right)^{1/2}.
\end{eqnarray*}
\end{proof}

\subsection{Trees are essentially regular}

While the work of \cite{sun} and \cite{gromov-schoen} together imply essential regularity for trees, it is important to our induction argument that the constants $c, \beta$ that appear in verifying essential regularity are independent of the target tree. For this reason, we provide below more details than can found in \cite[proof of Theorem 5.5]{gromov-schoen}. Given a Lipschitz Riemannian domain $\Omega$, let $$A_{\Omega}:=\{u:\Omega\to T\,|\, u \text{ is a harmonic map to a tree }T\}.$$
Notice in this definition (cf Definition \ref{def:aomegaw}) we suppress the $W$ from the subscript as every tree has the same finite reflection group.
\begin{proposition}\label{prop:tessreg}
Trees are essentially regular (cf~Definition~\ref{def:essreg}). 
\end{proposition}
\begin{proof}
Let $\Omega$ be a Riemannian domain and $K\subset\subset\Omega$ a compact subset.  Let $\tau_0=\min \{\frac 12, d(\partial\Omega,K)\}$.  We will prove that there exists $\theta \in (0, 1/2]$ such that for any harmonic map to a tree $u:\Omega\to T$, and any ball $B_\sigma(x_0)$ for $0<\sigma\leq\tau_0$ with $x_0\in K$, we have
\[
R^{u}(x_0,\theta\sigma)\leq\frac\theta2R^{u}(x_0,\sigma).
\]
An application of Lemma \ref{lem:intdifftheta} then implies that all trees are essentially regular, with the constants $c, \beta$ depending on $\Omega, K$ but not on the tree.

We proceed by contradiction.  Assume that there is no such $\theta$.  Thus, there exists  a sequence of harmonic maps $u_k:(\Omega,g)\to (T_k,d_k)$ to trees,  $x_k\in K$, $\sigma_k\in(0,\tau_0]$, and $\theta_k \to 0$ such that
\[
\frac{\theta_k}2R^{u_k}(x_k,\sigma_k) <R^{u_k}(x_k,\theta_k\sigma_k).
\]

Choose normal coordinates about each $x_k$ and then rescale the maps to obtain a new sequence $v_k:(B_2(0),g_k)\to (T_k,d_k')$ by taking $v_k(x):= u_k(\sigma_k x)$ and setting $d_k' := \mu_k d_k$ where we choose $\mu_k$ so that 
\begin{equation} \label{eq:sup=1}
\sup_{x \in B_1(0)}d_k'(v_k(x),v_k(0))=1.
\end{equation}
  For the rescaled sequence, we have the inequality
\begin{equation}\label{eq:badineqessreg}
\frac{\theta_k}2R^{v_k}(0,1) <R^{v_k}(0,\theta_k).
\end{equation}
Assume $\liminf_{k \rightarrow \infty} I^{v_k}(1/2) =0$, and choose a subsequence $v_{k_j}$ with $k_j\rightarrow \infty$  such that $I^{v_{k_j}}(1/2) \rightarrow 0$.  By applying the  monotonicity of $r \mapsto e^{cr} r^{-(n+1)}I^{v_{k_j}}(r)$ multiple times,  we have
\begin{eqnarray*}
\left(R^{v_{k_j}}(0,1) \right)^2 & \leq & \left(\frac{2}{\theta_{k_j}}R^{v_{k_j}}(0,\theta_{k_j}) \right)^2 \ \leq \ \frac{4}{\theta_{k_j}^2}\sup_{x \in B_{\theta_{k_j}}(0) }d_{k_j}'^2(v_{k_j}(x), v_{k_j}(0)) 
\\
& \leq & \frac{4e^{c\theta_{k_j}}}{\theta_{k_j}^{n+2}} \sup_{x \in B_{\theta_{k_j}}(0) } \int_{B_{\theta_{k_j}}(x)} d_{k_j}'^2(v_{k_j}, v_{k_j}(0)) \, d\mu_{g_{k_j}}
\\ 
& \leq & 
\frac{e^{c \theta_{k_j}}  2^{n+4}}{(2\theta_{k_j})^{n+2}} \int_{B_{2\theta_{k_j}}(0)} d_{k_j}'^2(v_{k_j}, v_{k_j}(0)) \, d\mu_{g_{k_j}}
\\
&  \leq &
\frac{e^{2c \theta_{k_j}} 2^{n+4}I^{v_{k_j}}(2\theta_{k_j})}{(2\theta_{k_j})^{n+1}} 
\ \leq \  e^{c /2} 2^{2n+5}I^{v_{k_j}}(1/2) \rightarrow 0.
\end{eqnarray*}
Thus, $R^{v_{k_j}}(0,1)\rightarrow 0$.  Combined with $I^{v_{k_j}}(1/2) \rightarrow 0$, we have $d'_{k_j}(v_{k_j},v_{k_j}(0)) \rightarrow 0$  uniformly in $B_1(0)$, contradicting (\ref{eq:sup=1}).  
Therefore, $I^{v_k}(1/2)\geq \epsilon>0$ for all $k$ sufficiently large.  Additionally, by Lemma~\ref{lem:ReversePoincare} and (\ref{eq:sup=1}), we also have a uniform energy bound, $E^{v_k}[B_1(0)] \leq E_0$.    By 
the convergence results of \cite[Section 3]{korevaar-schoen2} and applying Remark \ref{KLTheorem5.1}, the sequence $v_k|B_{1}(0)$ converges locally uniformly in the pullback sense to a harmonic map $v_\omega:B_{1}(0) \rightarrow T$ where $T$ is a tree.  The uniform boundedness  (\ref{eq:sup=1}) and the local uniform convergence  $d(v_k, v_k(0)) \rightarrow d_\omega(v_\omega, v_\omega(0))$ implies $d(v_k, v_k(0))\rightarrow d_\omega(v_\omega, v_\omega(0))$ in $L^2(B_{1}(0))$. 
Thus, the inequality $I^{v_k}(1/2) \geq \epsilon$ 
implies that $I^{v_\omega}(1/2)\geq \epsilon$; in particular $v_\omega$ is a non-constant  map.  We can now follow the last paragraph of the proof of \cite[Theorem 5.5]{gromov-schoen}  with help from \cite[Theorem 3.8 and Theorem 3.9]{sun} to account for the fact that the trees considered here are not necessarily locally finite. 
\end{proof}
\subsection{Extending to higher dimension}
\label{sec:essreg}

The following lemma is useful for finding the order gap, because it provides information about the order of product maps.
\begin{lemma}\label{lem:homord}
Let $X,Y$ be complete CAT(0) spaces, and let $u=(u_1,u_2)\colon B_r(0)\subset\R^n\to X\times Y$ be a homogeneous harmonic map of order $\alpha$ in the sense that $\text{Ord}^u(0,s)=\alpha$ for all $s\in(0,r)$.  Then if $u_1$ (resp. $u_2$) is nonconstant, it also satisfies $\text{Ord}^{u_1}(0,s)=\alpha$ for all $s\in(0,r)$.
\end{lemma}

In the proof, we will need following characterization of homogeneous harmonic maps.

\begin{lemma}\label{lem:homchar}
Let $X$ be a complete CAT(0) space.  A nonconstant harmonic map $v:{B_r(0)} \subset \R^n \to X$ has $\text{Ord}^v(x_0,\sigma)= \alpha$ for all $\sigma \in (0,r)$ if and only if it satisfies:
 \begin{enumerate}
 \item for all $\lambda<1$ and all $x\in B_r(0)$,
 \[
 d(v(\lambda x),v(0))=\lambda^\alpha d(v(x),v(0)),
 \]
 \item for each $x\in B_r(0)$, the map $\lambda\mapsto u(\lambda x)$ is a geodesic in $X$, and
 \item for all $\lambda<1$ and $\sigma\in(0,r)$, $E^v(\lambda \sigma)=\lambda^{n+2\alpha-2}E^v(\sigma)$.
 \end{enumerate}
\end{lemma}

Lemma \ref{lem:homchar} follows by work similar to that in \cite[Section 2]{gromov-schoen} for locally finite complexes, but can be extended to general complete CAT(0) spaces by the same adaptations which allowed us to define the order function.

\begin{proof}[Proof of Lemma \ref{lem:homord}]
We use the characterization of Lemma \ref{lem:homchar}, so we shall verify that (1), (2), and (3) of that Lemma hold for the maps $u_1,u_2$.  We write $d$ for the metric on $X\times Y$, and $d_1,d_2$ for the metrics on $X,Y$, respectively. Moreover, throughout we presume that both $u_1, u_2$ are nonconstant since, for example, if $u_1$ is constant then $u_2$ must be homogeneous of order $\alpha$.

We first verify (2). 
 To see this, note that the map $\lambda\mapsto u(\lambda x)$ is a geodesic in $X \times Y$ and, for $\gamma:=(\gamma_1,\gamma_2):[0,1]\to X\times Y$, $\gamma$ is a geodesic if and only if the two factors $\gamma_1,\gamma_2$ are.

Since it will be useful to us in the next point, we observe at this stage that since $\gamma_i(\lambda):= u_i(\lambda x)$ is a constant-speed geodesic, if we write $\delta_i(p)=d_i(p,u_i(0))$, then 
\[
\delta_1(\gamma_1(\lambda))\delta_2(\gamma_2(1))=\delta_1(\gamma_1(1))\delta_2(\gamma_2(\lambda)).
\]
In particular, the ratio between $\delta_1(u_1(\lambda x))$ and $\delta_2(u_2(\lambda x))$ does not depend on $\lambda$, but only on $x$.

Now, we verify (1).  Let $x\in B_r(0)$ and $\lambda<1$.  Since $u$ is homogeneous harmonic, we have that
\begin{equation}\label{eq:hom-ident-prod}
d(u(\lambda x),u(0))=\lambda^\alpha d(u(x),u(0)).
\end{equation}
Now, for $i,j \in \{1,2\}$, $i \neq j$, if  $u_i(x)=u_i(0)=0$, by the previous work we see that $u_i(\lambda x)=u_i(0)$ as well.  Therefore, because $u$ maps into a product, we have that
\[
d_j(u_j(\lambda x),u_{j}(0))=d(u(\lambda x),u(0))=\lambda^\alpha d(u(x),u(0))=\lambda^\alpha d_{j}(u_{j}( x),u_{j}(0))
\]
and the corresponding identity clearly holds for $u_i$. We now assume that $u_i(x)\neq u_i(0)$ for $i \in \{1,2\}$.  

Recall that by the definition of the product metric, $d^2((p_1,p_2),u(0))=d^2_1(p_1,u_1(0))+d_2^2(p_2,u_2(0))$.  Squaring \eqref{eq:hom-ident-prod} and expanding the result, we see that
\[
d^2_1(u_1(\lambda x),u_1(0))+d_2^2(u_2(\lambda x),u_2(0))=\lambda^{2\alpha}d^2_1(u_1(x),u_1(0))+\lambda^{2\alpha}d^2_2(u_2(x),u_2(0)).
\]
And since $\frac{d_1(u_1(\lambda x),u_1(0))}{d_2(u_2(\lambda x),u_2(0))}=c$ for $c$ independent of $\lambda$, 
\[
(1+c^2)d_1^2(u_1(\lambda x),u_1(0))=(1+c^2)\lambda^{2\alpha}d_1^2(u_1(x),u_1(0)).
\]
It follows that for $i=1,2$,
\begin{equation}\label{eq:hom-ident-con}
d_i(u_i(\lambda x),u_i(0))=\lambda^\alpha d_i(u_i(x),u_i(0)).
\end{equation}

Finally, we verify (3).  First, we note that because of \eqref{eq:hom-ident-con}, we can use triangle comparison to conclude that for $i=1,2$ and any $x\in B_r(0)$,
\[
|\nabla u_i|^2(\lambda x)\leq \lambda^{2\alpha-2}|\nabla u_i|^2(x).
\]
Integrating, we conclude that for $i=1,2$, and for any $\lambda<1$ and $\sigma\in(0,r)$
\begin{equation}\label{eq:hom-en-ineq}
E^{u_i}(\lambda\sigma)\leq\lambda^{n+2\alpha-2}E^{u_i}(\sigma).
\end{equation}

However, we also observe that for any $s$,
\[
E^u(s)=E^{u_1}(s)+E^{u_2}(s),
\]
so we expand
\[
E^u(\lambda\sigma)=E^{u_1}(\lambda\sigma)+E^{u_2}(\lambda\sigma)\leq\lambda^{n+2\alpha-2}(E^{u_1}(\sigma)+E^{u_2}(\sigma))=\lambda^{n+2\alpha-2}E^u(\sigma).
\]
Since the outer two quantities are equal, the inner inequality cannot be strict.  Therefore, for each $i=1,2$ we have
\[
E^{u_i}(\lambda\sigma)=\lambda^{n+2\alpha-2}E^{u_i}(\sigma).
\]

Therefore both $u_1,u_2$ are in fact homogeneous harmonic maps of order $\alpha$ in the sense that $\text{Ord}_{u_i}(0,s)=\alpha$ for all $s\in(0,r)$.
\end{proof}

\begin{proof}[Proof of Theorem \ref{thm:ordgap}]
Throughout, we proceed by induction on $N$, the dimension of the building; we suppose at all stages that we have proved all of these results for all buildings of dimension less than $N$.

First, note that for trees, (1) is proven in \cite[Theorem 3.8]{sun} and the constant $\epsilon$ is independent of the tree. Now suppose $N \geq 2$. To show (1), suppose we have a sequence of harmonic maps $u_k:(\Omega,g)\to X_k$ into Euclidean buildings of type $W$, and $\text{Ord}^{u_k}(p)=\alpha_k\to1$ as $k\to\infty$. We replace each $u_k$ by its homogeneous harmonic ultralimit $(u_k)_\omega$ using the (Gromov-Schoen) blow up maps and Remark \ref{KLTheorem5.1}. 
So now we have a sequence of homogeneous harmonic maps $(u_k)_\omega:\R^n \to (X_k)_\omega$, where $(X_k)_\omega$ are Euclidean buildings of type $W$, and $\text{Ord}^{(u_k)_\omega}(0,1)=\alpha_k$ by homogeneity. Now for each $k$, rescale the metric $d_k':= \mu_k (d_k)_\omega$ so that $\sup_{B_1(0)} d_k'((u_k)_\omega,(u_k)_\omega(0))=1$. Then by Lemma \ref{lem:ReversePoincare}, $E^{(u_k)_\omega}[B_{\frac 12}(0)]\leq C$.

For ease of notation, relabel the ultralimit maps as $u_k$ and the target buildings as $(X_k,d_k')$. Now apply the convergence results of \cite[Section 3]{korevaar-schoen2} and Remark \ref{KLTheorem5.1} to this new sequence of maps and let $ u_\omega:= \wlim u_k$ and $X_\omega:= \wlim X_k$. Since $E^{u_k}[B_{\frac 14}(0)]\to E^{ u_\omega}[B_{\frac 14}(0)]$ and $\lim_{k \to \infty}I^{u_k}(1/4)=I^{ u_\omega}(1/4)$, we have $\text{Ord}^{u_k}(0,1/4)\to\text{Ord}^{ u_\omega}(0,1/4)$. And thus, $\text{Ord}^{ u_\omega}(0,1/4)=1$.  By the monotonicity of order, for all $0<r<1$, $\text{Ord}^{ u_\omega}(0,r)=1$ and thus $ u_\omega$ is a homogeneous harmonic map of degree $1$. 

In particular, by Proposition \ref{prop:flats}, there exists $r_0>0$ such that $u_\omega|B_{r_0}$ can be extended to a homogeneous harmonic map $L:\R^n \to X_\omega$ which has image contained in some apartment $A_\omega$ of $ X_\omega$ where $A_\omega=\wlim A_k$ for some sequence of apartments $A_k\subseteq X_k$. Constructing the maps $L_k$ as in \eqref{eq:Lk}, each $(X_k ,A_k,L_k)$ is an $(X_\omega,A_\omega,L)$-triple, and following the proof of Lemma \ref{lem:ukLk} we observe that there exists a subsequence of the $u_k$ such that
\[
\sup_{B_{r_0}(0)}d_k(u_k,L_k)\to0.
\]

\newcommand{\Ord}{\text{Ord}}
For each $k$, consider the subbuilding $P_{F_k}=F_k\times Y_k$ consisting of all flats parallel to $F_k:= L_k(\R^n)$.  We observe that this complex is essentially regular since $F_k \simeq \R^m$ and $\text{dimension}(Y_k)<N$, so Theorem \ref{GS Theorem 5.1} applies.  In particular, for sufficiently large $k$, $u_k(B_{\frac{r_0}4})\subset F_k\times Y_k$. So, write $u_k|B_{\frac{r_0}4}(0)=(u_{k,1},u_{k,2})$ and observe that this is a map into a product of NPC spaces, where $Y_k$ is a building of type $W_k$. Since each possible $W_k$ that can occur is a subgroup of a restriction of the original $W$, there are only finitely many possible $W_k$, all depending on $W$. By Lemma \ref{lem:homord} we note that since $u_k$ is homogeneous of order $\alpha_k$, the maps $u_{k,1}:B_{\frac{r_0}4}\to F_k$ and $u_{k,2}:B_{\frac{r_0}4}\to Y_k$ are either constant, or homogeneous of order $\alpha_k>1$. Because $F_k$ is Euclidean, if $u_{k,1}$ is nonconstant, $\Ord^{u_{k,1}}(0)\geq2$, and by the inductive hypothesis, if $u_{k,2}$ is nonconstant, $\Ord^{u_{k,2}}(0)\geq1+\epsilon_{W_k}$. Since $u_k$ is nonconstant, $u_{k,1}$ and $u_{k,2}$ cannot both be constant, so $\Ord^{u_k}(0)\geq\min\{2,1+\epsilon_{W_k}\}$.  Now, since there are only finitely many possible $W_k$ that can occur, for any large enough $k$, we have $\Ord^{u_k}(0)\geq1+\epsilon_W$, where $\epsilon_W:=\min\{\epsilon_{W_k}\}$.  But this means that the $\alpha_k$ certainly cannot converge to $1$.

For (2), assume that $\text{Ord}^u(p)=1$ and let $u_\sigma$ be the (Gromov-Schoen) blow up maps at $p$  defined by \eqref{eq:usigmamap}.  Then $^{g_\sigma}E^{u_\sigma}(1) \leq 2$ for $\sigma>0$ sufficiently small.  Let $\sigma_k \rightarrow 0$, $r_0 \in (0,1)$, $X_k$, $A_k$, $L_k$, $X_\omega$, $A_\omega$, $L_\omega$ be given by Proposition~\ref{LinearApprox}.    Let $\delta_0=\delta_0(E_0,r_0,X_\omega,A_\omega,L_\omega)>0$ be chosen as in Theorem~\ref{GS Theorem 5.1}.  By Proposition~\ref{LinearApprox}, we can choose $k$ such that  $\sup_{x \in B_{r_0}(0)} d_k(u_k(x),L_k(x))<\delta_0$.  By the inductive hypothesis, $P_{F_k}\simeq F_k \times Y_k$ is essentially regular since $Y_k$ is a Euclidean building of dimension $N\text{-}m$ and $F_k \simeq \R^m$ for some $m\geq 1$. Thus, for this $k$ Theorem \ref{GS Theorem 5.1} implies that $u_k(B_{\frac{r_0}4}(0)) \subset P_{F_k}\simeq \R^m  \times Y_k$.  

Now, let  $\rho:=\sigma_k r_0/4$.   If $m=N$, then $F_k$ is an apartment and we are done since  $u|{B_\rho(p)} \subset F_k$.  Otherwise, $u|{B_\rho(p)}$ decomposes as two harmonic maps $\overline {u_1}:B_\rho(p) \to F_k \simeq \R^{m}$ and $\overline{ u_2}:B_\rho(p) \to Y_k$. Note that $\overline{u_1}$ is full rank by the construction of $F_k$ so if $\text{Ord}^{\overline{u_2}}(p)\geq 1+\epsilon$ or $\overline {u_2}$ is constant, then we are done.  
Alternatively, if $\text{Ord}^{\overline{u_2}}(p)=1$, then we use the inductive hypothesis again to assert that there exists an $r \in (0,\rho)$ and a subbuilding of $Y_k$, isometric to $\R^{j} \times Z$, where $j \in \{1, \dots, \min\{n, N-m\}\}$, $Z$ is a building of dimension $N-m-j$ and $\overline{ u_2}|B_r(p)$  decomposes as $\widehat{ u_1}:B_r(p) \rightarrow \R^{j}$ and ${ u}_2:B_r(p) \rightarrow Z$ where $\widehat {u_1}$ is full rank and $\text{Ord}^{{u}_2}(p)\geq 1+\epsilon$ or $u_2$ is constant. Then $u|B_r(p):=(u_1, u_2):B_r(p) \to \R^{m+j} \times Z$ where $u_1=(\overline{u_1}, \widehat{u_1}):B_r(p) \to \R^{m+j}$ is full rank.

For (3), we initially follow the outline of the proof for trees given in Proposition \ref{prop:tessreg}. By contradiction, we again assume there exists  a sequence of harmonic maps $u_k:(\Omega,g)\to (X_k,d_k)$ where each $X_k$ is a building of dimension $N$ of type $W$,  $x_k\in K$, $\sigma_k\in(0,\tau_0]$, and $\theta_k \to 0$ such that
\[
\frac{\theta_k}2R^{u_k}(x_k,\sigma_k) <R^{u_k}(x_k,\theta_k\sigma_k).
\]
Rescaling as in Proposition \ref{prop:tessreg}, we produce a sequence $v_k:(B_2(0),g_k)\to (X_k,d_k')$ by taking $v_k(x):= u_k(\sigma_k x)$ and setting $d_k' := \mu_k d_k$ where we choose $\mu_k$ so that 
\begin{equation} \label{eq:sup=1}
\sup_{x \in B_1(0)}d_k'(v_k(x),v_k(0))=1.
\end{equation}
  For the rescaled sequence, we have the inequality
\begin{equation}\label{eq:badineqessreg2}
\frac{\theta_k}2R^{v_k}(0,1) <R^{v_k}(0,\theta_k)
\end{equation}and as before, $\wlim v_k=v_\omega:(B_1(0),g) \to (X_\omega, d_\omega)$ is a harmonic map where $X_\omega$ is of type $W$. Now we consider two cases, depending on the order of $v_\omega$ at $0$.

\emph{Case 1:} Presume that $\text{Ord}^{v_\omega}(0)>1$. Then by part (1) of this same lemma, which we already know holds for buildings of dimension $N$, $\text{Ord}^{v_\omega}(0)\geq 1+ \epsilon$, where $\epsilon$ depends on $W$. Thus there exists a constant $c$ such that
\[
R^{v_\omega}(0, \sigma) \leq c \sigma^{1+\epsilon}R^{v_\omega}(0,1)
\] and for small enough $\theta$ depending on $c,\epsilon$, we contradict \eqref{eq:badineqessreg2} for $k$ large enough.

\emph{Case 2:} Presume that $\text{Ord}^{v_\omega}(0)=1$ (and $v_\omega$ is not homogeneous since otherwise the contradiction to \eqref{eq:badineqessreg2} is immediate). Then by part (2) of this same lemma, there exists a subbuilding $\FY$ and a radius $\sigma_0>0$ so that $v_\omega(B_{\sigma_0}(0)) \subset \FY$. Now $F \simeq \R^m$ and $Y$ is a building of dimension less than $N$ so by induction and \cite[Lemma 6.1]{gromov-schoen}, $\FY$ is essentially regular. Thus there exist $c, \beta$ depending on $K, \Omega, W$ such that $R^{v_\omega}(0,\sigma) \leq c \sigma^{1+\beta}R^{v_\omega}(0,\sigma_0)$. By the convergence of $v_k$ to $v_\omega$ and the monotonicity of $R$ in the second argument, for large enough $k$, depending on $c, \beta$, we get a contradiction to \eqref{eq:badineqessreg2}.
\end{proof}


\begin{thebibliography}{ABC}
\label{references}

\bibitem[ABCKT]{ABCKT} J.~Amor\'os, M.~Burger, K.~Corlette, D.~Kotschick, and D.~Toledo. Fundamental groups of compact K\"ahler manifolds. Vol. 44. Mathematical Surveys and Monographs.
American Mathematical Society, Providence, RI, 1996.

\bibitem[AHL]{ahl} R. Appenzeller, A. H\'ebert, and A. Lytchak. {\it Flat subsets of Euclidean buildings.} arXiv:2512.11658v1
\bibitem[BF]{bader-furman} U.~Bader and A. ~Furman. {\it An extension of Margulis’s superrigidity theorem.} Dynamics. Geometry, Number Theory--The Impact of Margulis on Modern Mathematics, pages 47-65. University Chicago Press, Chicago, IL (2022).

\bibitem[BH]{bridson-haefliger} M.~Bridson and A.~Haefliger.  {\it Metric Spaces of Non-Positive Curvature}. Springer-Verlag, Berlin, Heidelberg, New York, 1999.

\bibitem[BT]{bruhat-tits} 
F.~Bruhat, François and J.~Tits, Jacques. {\it Groupes r\'eductifs sur un corps local, I. Donn\'ees radicielles valu\'ees}. Publ.~Math.~IHES 41 (1972) 5-251.

\bibitem[BDDM]{bddm}
D.~Brotbek, G.~Daskalopoulos, Y.~Deng, and C.~Mese.
{\it Pluriharmonic maps into buildings and symmetric differentials} arXiv:2206.11835 

\bibitem[CL]{caprace-lytchak} P-E. Caprace and A. Lytchak. {\it At infinity of finite-dimensional CAT(0) spaces}. Math. Ann. 346 (2010) 1-21.
\bibitem[CM]{monodcarp} P-E. Caprace and N. Monod. {\it Isometry groups of non-positively curved spaces: structure theory.}
Journal of Topology 2 (2009) 661-700.

\bibitem[CT]{carlson-toledo} J.A.~Carlson and D.~Toledo.  {\it Harmonic mappings of Kähler manifolds to locally symmetric spaces.} Publications Math\'{e}matiques de l'IH\'{E}S 69 (1989) 173-201. 
\bibitem[C]{corlette} 
\newblock K. Corlette.  {\em Archimedean superrigidity and hyperbolic geometry.} 
\newblock  Ann. of Math. 135 (1992) 165-182. 

\bibitem[CS]{corlette-simpson} K.~Corlette and C.~Simpson. {\it On the classification of rank-two representations of quasiprojective fundamental groups.} Compositio Mathematica 114 (2008) 1271-1331.


\bibitem[DDW]{ddw} G.~Daskalopoulos, S.~Dostoglou, and R.~Wentworth.  {\it Character varieties and
harmonic maps to
$\R$-trees} Math.~Res.~Lett. 5 (1998) 523-533.



\bibitem[DM1]{daskal-meseMem} G.~Daskalopoulos and C.~Mese. {\it On the Singular Set of Harmonic Maps into DM-Complexes.}  Memoirs of the AMS, Volume 239 Number 1129  (2016).


\bibitem[DM2]{daskal-meseINV}
G. Daskalopoulos, C. Mese.  {\it Rigidity of Teichm\"uller Space.}  Invent.~Math. 224 (2021) 791- 916.


\bibitem[DM3]{daskal-meseMRL}
G. Daskalopoulos, C. Mese. 
{\it Uniqueness of equivariant harmonic maps to symmetric spaces and buildings.} To appear Math.~Res.~Lett.


\bibitem[DMV]{daskal-meseGAFA} G.~Daskalopoulos, C.~Mese and A.~Vdovina. {\it Superrigidity of Hyperbolic Buildings.} GAFA 21 (2011) 1-15.

\bibitem[D]{dees}B. K. Dees. {\it Rectifiability of the singular set of harmonic maps into buildings.} J.~Geom.~Anal. 32 (2022), no.7, Paper No. 205, 57.

\bibitem[F]{freidin}B. Freidin, {\it A Bochner formula for harmonic maps into non-positively curved metric spaces}, Calc. Var. Partial Differential Equations {\bf 58} (2019), no.~4, Paper No. 121, 28 pages

\bibitem[FZ]{freidin-zhang}B. Freidin and Y. Zhang, {\it A Liouville-type theorem and Bochner formula for harmonic maps into metric spaces}, Comm. Anal. Geom. {\bf 28} (2020), no.~8, 1847--1862.

\bibitem[GS]{gromov-schoen} M. Gromov and R. Schoen. {\it Harmonic maps into singular
spaces and $p$-adic superrigidity for lattices in groups of rank
one.} Publ. Math. IHES 76 (1992)  165-246.


\bibitem[J]{jost}J. Jost. {\it Equilibrium maps between metric spaces.} Calc.~Var. 2 (1994) no. 2, 173-204.

\bibitem[KNPS1]{katzarkov-et.al}
L.~Katzarkov, A.~Noll, P.~Pandit, C.~Simpson. {\it Harmonic maps to buildings and singular perturbation theory.} Commun.~Math.~Phys. 336 (2015) 853-903 (2015).

\bibitem[KNPS2]{katzarkov-et.al2} 
L.~Katzarkov, A.~Noll, P.~Pandit, C.~Simpson.
{\it Constructing buildings and harmonic maps}. 
Algebra, Geometry, and Physics in the 21st Century. Progress in Mathematics, vol 324. Birkhäuser, 2017.
D.~Auroux, L.~Katzarkov, T.~Pantev, Y.~Soibelman, Y.~Tschinkel (editors) 

\bibitem[Ki]{kim} S.~Kim.  {\it Harmonic Maps and the Moduli of Higgs Bundles.} Ph.D.~Thesis, Brown University 2017.  
https://repository.library.brown.edu/studio/item/bdr:733382/

\bibitem[KL]{kleiner-leeb}  B.~Kleiner and B.~Leeb.  {\it Rigidity of quasi-isometries for symmetric spaces and Euclidean buildings}.
Publications Math\'matiques de l'IH\'ES 86 (1997) 115-197. 

 \bibitem[KS1]{korevaar-schoen1} N.~Korevaar and R.~Schoen.  {\it Sobolev spaces and harmonic maps for metric space targets}.
Comm.~Anal.~Geom. 1 (1993) 561-659.


\bibitem[KS2]{korevaar-schoen2}  N. Korevaar and R. Schoen.  {\it
Global existence theorem for harmonic maps to non-locally compact spaces.}  Comm.  Anal. Geom. 5 (1997), 333-387.

\bibitem[KS3]{korevaar-schoen3}  N. Korevaar and R. Schoen.  {\it
Global existence theorem for harmonic maps:  finite rank spaces and an approach to rigidity for smooth actions.} Preprint.


\bibitem[L]{lytchak} A.~Lytchak.  {\it Differentiation in metric spaces.} St~ Petersburg~Math.~J. 16 (2005) 1017–104.

\bibitem[LTW]{loftin-tamburelli-wolf} J.~Loftin, A.~Tamburelli, M.~Wolf. {\it Limits of cubic differentials and buildings.} arXiv:2208.07532

\bibitem[Ma]{margulis} G.~Margulis.  {\it Discrete subgroups of semisimple Lie groups.} Ergebnisse der Mathematik
und ihrer Grenzgebiete (3) [Results in Mathematics and Related Areas (3)], vol. 17, Springer-
Verlag, Berlin, 1991.

\bibitem[P1]{parreau} A.~Parreau.  {\it
Immeubles affines: construction par les normes et\'etude des isom\'etries.} Contemporary Math. 262
(2000) 263-302.

\bibitem[P2]{parreau2} A.~Parreau.  {\it Compactification d’espaces de repr\'sentations de groupes de type fini. }Math. Z. 272 (2012) 51-86. 

\bibitem[Si]{siu}  Y.T. Siu.  {\it The complex analyticity of
harmonic maps and the strong rigidity of compact K\"{a}hler
manifolds.}  Ann. of Math. 112 (1980) 73-111.

\bibitem[Su]{sun} X.~Sun.  {\it Regularity of Harmonic Maps to Trees}.
American Journal of Mathematics 125 (2003) 737-771.

\bibitem[T]{tits} J.~Tits.  {\it Immeubles de type affine},  in Buildings and the geometry of diagrams, Como Italy 1984, L.A. Rosati
editor, Springer lecture note in Math. 1181 (1986) 159-190.

\bibitem[W]{wolf} M.~Wolf.  {\it Harmonic maps from surfaces to $\R$-trees.} Math.~Z. 218 (1995) 577-593.

\bibitem[ZZZ]{zhang-zhong-zhu}H.-C. Zhang, X. Zhong and X.-P. Zhu, {\it Quantitative gradient estimates for harmonic maps into singular spaces}, Sci. China Math. {\bf 62} (2019), no.~11, 2371--2400
\end{thebibliography}
\end{document}